\documentclass[10pt]{article}

\usepackage{graphicx}
\usepackage{subcaption}
\usepackage{sidecap}
\usepackage{amsfonts}
\usepackage{amsmath}
\usepackage{url}
\usepackage{amssymb}
\usepackage{cprotect}
\usepackage{amsthm}
\usepackage{rotfloat}
\usepackage{array}
\usepackage{float}
\usepackage{xspace}
\usepackage{multirow}

\usepackage{psfrag}
\usepackage{graphics}
\usepackage{float}

\usepackage{comment}

\usepackage{algorithm}
\usepackage{algpseudocode}
\usepackage{epstopdf}

\usepackage{soul}
\usepackage{xspace}
\usepackage{soul,color}

\usepackage[margin=1in]{geometry}

\newcommand{\OK}{\texttt{OK}\xspace}
\newcommand{\MAX}{\texttt{MAX}\xspace}
\newcommand{\OTHER}{\texttt{OTHER}\xspace}
\newcommand{\defeq}{\stackrel{\rm def}{=}}
\DeclareMathOperator*{\argmin}{arg\,min}
\newtheorem{theorem}{Theorem}[section]
\newtheorem{lemma}[theorem]{Lemma}

\newcommand{\RR}{{\mathbb R}}
\newcommand{\AAA}{{\mathcal A}}
\newcommand{\BBB}{{\mathcal B}}
\newcommand{\FFF}{{\mathcal F}}
\newcommand{\CCC}{{\mathcal C}}
\newcommand{\bAAA}{{\bar{\mathcal A}}}
\newcommand{\bFFF}{{\bar{\mathcal F}}}
\newcommand{\NNN}{{\mathcal N}}

\newcommand{\algoname}{NQN\xspace}

\newcommand{\Footnote}[1]{\footnote{#1}}
\newcommand{\nitish}[1]{{\color{blue} #1}}

\begin{document}
\title{A Limited-Memory Quasi-Newton Algorithm for Bound-Constrained Nonsmooth Optimization}

\author{Nitish Shirish Keskar\thanks{Email: keskar.nitish@u.northwestern.edu} \qquad  Andreas W{\"a}chter\thanks{Email: andreas.waechter@northwestern.edu} \\
Department of Industrial Engineering and Management Sciences,\\ Northwestern University, Evanston, Illinois, USA -- 60208}
           
\maketitle

\begin{abstract}
We consider the problem of minimizing a continuous function that may be nonsmooth and nonconvex, subject to bound constraints. We propose an algorithm that uses the L-BFGS quasi-Newton approximation of the problem's curvature together with a variant of the weak Wolfe line search.
The key ingredient of the method is an active-set selection strategy that defines the subspace in which search directions are computed.
To overcome the inherent shortsightedness of the gradient for a nonsmooth function, we propose two strategies.  The first relies on an approximation of the $\epsilon$-minimum norm subgradient, and the second uses an iterative corrective loop that augments the active set based on the resulting search directions.
We describe a Python implementation of the proposed algorithm and present numerical results on a set of standard test problems to illustrate the efficacy of our approach. 
\end{abstract}


\section{Introduction}
\label{sec:introduction}

We propose an algorithm for solving bound-constrained optimization problems of the form
\begin{align}
\label{eq:problem}
   \min_{x\in\RR^n} \ &  f(x) \\
    \text{s.t.} \ & l \leq  x \leq u, \nonumber
\end{align}
where the objective function $f:\mathbb{R}^n\longrightarrow \mathbb{R}$ is continuous but may not be differentiable everywhere.
%
No assumptions are placed on the convexity of $f$. The lower bounds $l\in(\RR\cup\{-\infty\})^n$ and upper bounds $u\in(\RR\cup\{\infty\})^n$ can take values of $-\infty$ or $\infty$ whenever the variables are unbounded in those coordinates.  We assume that the problem is feasible; i.e., $l\leq u$.

Many algorithms have been proposed for solving \eqref{eq:problem} when $x$ is unconstrained. Some of these methods include gradient-sampling methods \cite{burke2005robust,curtis2013adaptive,kiwiel2007convergence}, bundle methods \cite{lmbmnew,haarala2007globally,surveybundle}, quasi-Newton methods \cite{curtis2015quasi, kaku,Lewis2013,Lewis2015,skajaa2010limited}, and hybrid methods \cite{curtis2015quasi}. Gradient-sampling methods randomly sample gradients in the vicinity of the iterate to calculate an estimate of the minimum-norm subgradient. In conjunction with an Armijo-like line search, global convergence can be proved using these minimum-norm subgradients as search directions. Bundle methods aggregate subgradients from previous iterates and iteratively solve piecewise-quadratic approximations of the objective function to generate steps.  
Recently, Lewis and Overton \cite{Lewis2013} observed that the unadulterated BFGS method works very well when applied to unconstrained nonsmooth optimization problems so long as the weak Wolfe line search is performed. Skajaa \cite{skajaa2010limited} reported similar results for L-BFGS \cite{liu1989limited}. For problems ranging from $n=100$ to $n=10000$, it was found that L-BFGS was not only more efficient in solving test problems, but it was also more reliable compared to other methods. However, theoretical convergence guarantees (or the lack thereof) of (L-)BFGS for nonsmooth problems remain to be shown. Recent efforts (e.g., \cite{curtis2015quasi}) focused on the design of a hybrid strategy that retains the efficacy of standard L-BFGS but ensures convergence through a gradient-sampling approach. Other approaches for solving \eqref{eq:problem} include subgradient methods, quasi-secant methods, and discrete gradient methods. We refer the reader to \cite{bagirov2014introduction,clarke1990optimization,nsocomparison,surveybundle,skajaa2010limited} for a detailed summary of these methods and their numerical performance.

In certain applications, it is necessary to optimize a nonsmooth objective function subject to bound constraints. These include applications in many fields including statistics, optimal control, and as subproblems for certain robust optimization problems \cite{bagirov2014introduction}. Some of the algorithms described above can be extended to solve problems with bound constraints. For instance, the LMBM-B \cite{karmitsa2010adaptive,karmitsa2010limited} method extends the limited-memory bundle method to \eqref{eq:problem}. Gradient-sampling methods have also been extended to the case of constrained optimization \cite{curtis2012sequential}.
A natural question is whether the surprising and remarkable success of the unadulterated (L-)BFGS method in the unconstrained case can be extended to problems with bound constraints.

The L-BFGS-B method is a variant of L-BFGS for minimizing a smooth objective function over box constraints.  At an iterate $x^k$, the method first determines an active set by computing a Cauchy point $\tilde x^k$ as the first local minimizer $\alpha>0$ of $f$ along the gradient-projection path $\alpha\mapsto P(x^k-\alpha\nabla f(x^k))$.  Here, $P(v)$ is the orthogonal projection of a vector $v\in\RR^n$ onto the feasible hypercube $[l,u]$.  The bound constraints that are tight at the Cauchy point $\tilde x^k$  then define an active set, $\AAA^k=\{i : \tilde x^k_i \in \{l_i,u_i\}\}$, and a subspace step $\tilde p^k$ is computed as the solution of the problem
\begin{align}\label{eq:subprob}
\min_{p\in\RR^n} \ & f(x^k) + \nabla f(x^k)^Tp + \frac12 p^TB^kp \\
 \text{s.t.} \ & p_i = 0\text{ for all } i\in\AAA^k.\nonumber
\end{align}
The objective in this subproblem is a quadratic model of the original objective function.  Its Hessian matrix $B^k$ is defined by the L-BFGS update and therefore is positive definite.  
Subproblem \eqref{eq:subprob} is solved efficiently using the compact-form representation of L-BFGS \cite{byrd1994representations}.
 Finally, the overall step $p^k = (\tilde x^k+\tilde p^k) -x^k$ is computed and a projected line search is performed along the path $\alpha \mapsto P( x^k + \alpha p^k)$ to find a step size satisfying the strong Wolfe conditions.

Henao et al.~\cite{lbfgsbns} recently proposed L-BFGS-B-NS as a variant of L-BFGS-B for solving \eqref{eq:problem} with a nonsmooth objective function.  The only difference to the original method is that the strong Wolfe line search is replaced with the weak Wolfe line search.  This is the same modification that was suggested by Lewis and Overton \cite{Lewis2013} in the unconstrained case.


In this paper, we propose a different adaptation of the L-BFGS method.  Our method first determines an active set based on the bound constraints that are tight at the current iterate, without referring to a Cauchy point.  After computing the search direction from \eqref{eq:subprob}, a new iterate is determined using a variant of the weak Wolfe line search.

The key ingredients in our method are active-set selection strategies that take into account the non\-smoothness of the function.  
First, we propose the use of an approximation of the minimum-norm $\epsilon$-subgradient instead of the gradient to determine which bound constraints are binding at the current iterate.  Second, we explore an iterative corrective mechanism that augments the active set until the final search direction  points inside the feasible region. 
%


 Throughout the paper, we assume that the function is differentiable at each iterate and trial point, and that its gradient can be computed.  This is in line with the work by Lewis and Overton \cite{Lewis2013} who make the same assumption for their numerical experiments.
%
Burke et al.\ \cite{burke2005robust} describe a mechanism that perturbs a trial point in case $f$ is not differentiable at that point.
In the box-constrained case, the boundary of the feasible region might align with a manifold of nondifferentiability.  Then, the projections carried out during the line search might generate trial points that lie in this manifold.  To circumvent this problem, we assume that there is an extension of the function beyond the feasible region that is differentiable at almost all boundary points.  For example, consider the feasible set $[0,1]\subset\RR$ with objective function $f(x)=|x|$ which is not differentiable at the boundary point $x=0$.  When we replace the objective with $\tilde f(x)=x$, the objective values are identical within the feasible region, resulting in the same optimal solution, but the function is now differentiable at $x=0$. 

The paper is organized as follows. In the subsection to follow, we introduce some notation that is used throughout the paper. In Section 2, we describe our algorithm including the active-set prediction and correction strategies, as well as the proposed weak Wolfe line search. Finally, in Section 3, we present details of our Python implementation and detailed numerical results examining the efficacy of our approach. 

\subsection{Notation}
We use superscripts to denote the iteration index and subscripts to denote a specific element of a vector. For instance, $x^k_j$ refers to the $j^{th}$ element of the $k^{th}$ iterate.  We abbreviate $[\nabla f(x)]_i$ by $\nabla_i f(x)$. 
We define the instantaneous projection of a direction $p$ at an iterate $x$ as
\begin{equation}\label{eq:def_T}
[T(x,p)]_i = 
\begin{cases}
p_i & \text{ if } x_i \in (l_i,u_i) \\
\max(p_i,0) & \text{ if } x_i = l_i\\
\min(p_i,0) & \text{ if } x_i = u_i. 
\end{cases}
\end{equation}
This operator zeroes out those components of $p$ for which $x$ is at its bounds with $p$ pointing in the direction of infeasibility.
%
%
We use the notation $B_\epsilon(x)$ to denote the closed ball of radius $\epsilon$ centered at $x$. Further, we denote the cardinality of a set $A$ by $|A|$.
Finally, given a vector $v$ and a set of indices $\mathcal A$,  $v_{\mathcal A}$ refers to the subvector corresponding to the indices in $\AAA$.  Similarly, given a matrix $M$, then $M_{\AAA,\mathcal B}$ denotes the submatrix with row indices given in $\AAA$ and column indices given in $\mathcal B$.
Finally, we let $\NNN=\{1,\ldots,n\}$ be the set of all variable indices.

\section{Proposed Algorithm}

\subsection{Active-Set Framework}
\label{sec:activeset}

The proposed algorithm is an active-set method which, at each iteration, determines an estimate $\AAA^k$ of the optimal active set $\AAA^*:=\{i\in\NNN : x^*_i=l_i \text{ or } x_i^*=u_i\}$ of a local solution $x^*$ of \eqref{eq:problem}.  
We say that the bound constraints in $\AAA^*$ are tight at $x^*$.
%
The L-BFGS-B algorithm chooses as active set $\AAA^k$ the bounds at which the Cauchy point is tight.  
In contrast, our method chooses from bounds for which the current iterate $x^k$ itself is tight, without referring to a Cauchy point.


For a smooth objective function, we might consider the set of the tight constraints that are binding; i.e.,
\begin{equation}
\label{eq:activeset}
    \mathcal{B}^k(g^k) = \{i\in\NNN:x^k_i=l_i \text{ and } g^k_i\geq0\} \cup \{i\in\NNN:x^k_i=u_i \text{ and } g^k_i\leq0\}
\end{equation}
with $g^k = \nabla f(x^k)$.  These are the coordinates for which the gradient predicts no decrease in the objective if the corresponding components of the iterate are moved inside the feasible region.  With this, $\nabla_i f(x^k)=0$ for all $i\in\NNN\setminus\BBB(\nabla f(x^k))$ if and only if $x^k$ satisfies the first-order optimality conditions for problem \eqref{eq:problem} at $x^k$; i.e., 
\begin{align}
\nabla_i f(x^k)=0 & \text{ for all $i$ with } l_i < x^k_i < u_i \nonumber\\
\nabla_i f(x^k)\geq 0 & \text{ for all $i$ with } x^k_i = l_i \label{eq:firstorder}\\
\nabla_i f(x^k)\leq 0 & \text{ for all $i$ with } x^k_i = u_i.\nonumber
\end{align}
Consequently, the subspace step $p^k$ obtained from solving \eqref{eq:subprob} with $\AAA^k=\BBB(\nabla f(x^k))$ is zero if and only if $x^k$ is a first-order optimal point.  
In addition, it can be shown that $p^k$ is a descent direction for the projected line search; i.e., the function $\alpha\mapsto f(P(x^k+\alpha p^k))$ is decreasing for $\alpha>0$ sufficiently small.  (This is a consequence of \cite[Proposition 1]{bertsekas1982projected}.)

Consider a simple algorithm that computes search directions from \eqref{eq:subprob} with $\AAA^k=\BBB(\nabla f(x^k))$ and performs a projected line search to determine the new iterate $x^{k+1}=P(x^k+\alpha^kp^k)$ with some step size $\alpha^k>0$.  Suppose that $f$ is differentiable and that the iterates converge to a non-degenerate first-order optimal $x^*$; i.e., $x^*$ satisfies \eqref{eq:firstorder} and $\nabla_i f(x^*)\neq 0$ for all $i\in\AAA^*$.  Further assume that at some iterate $x^k$ sufficiently close to $x^*$, the bounds that are tight at $x^k$ are identical to the optimal active set; i.e., $\{i\in\NNN : x_i^k=l_i \text{ or }x_i^k=u_i  \}=\AAA^*$. It is then not difficult to show that $\AAA^k=\BBB(\nabla f(x^*))$ for all large $k$.  In other words, the optimal active set is identified in a finite number of iterations.  This observation motivates the choice $\AAA^k=\BBB(\nabla f(x^k))$.

The conclusion in the previous paragraph was drawn under the strong assumption that an iterate is encountered at which all constraints in $\AAA^*$ are tight.  To the best of our knowledge, no convergence proof has been established for the simple algorithm when this assumption is lifted, even when $f$ is differentiable.  
%
Nevertheless, despite the lack of theoretical convergence guarantees, our proposed active-set selection strategy is based on \eqref{eq:activeset} since it seems to perform well in practice in our setting. Recall that global convergence has not been proved for the unadulterated L-BFGS algorithm with a nonsmooth objective function even in the unconstrained case.

In the context of nonsmooth optimization, the gradient of the objective function can be very myopic in regions close to a manifold on which the function is nondifferentiable, and a gradient-based active-set identification can be quite misleading. We illustrate this with a simple example, depicted in Figure~\ref{fig:myopism}.
\begin{figure}[t]
    \centering
    \includegraphics[width=0.5\textwidth]{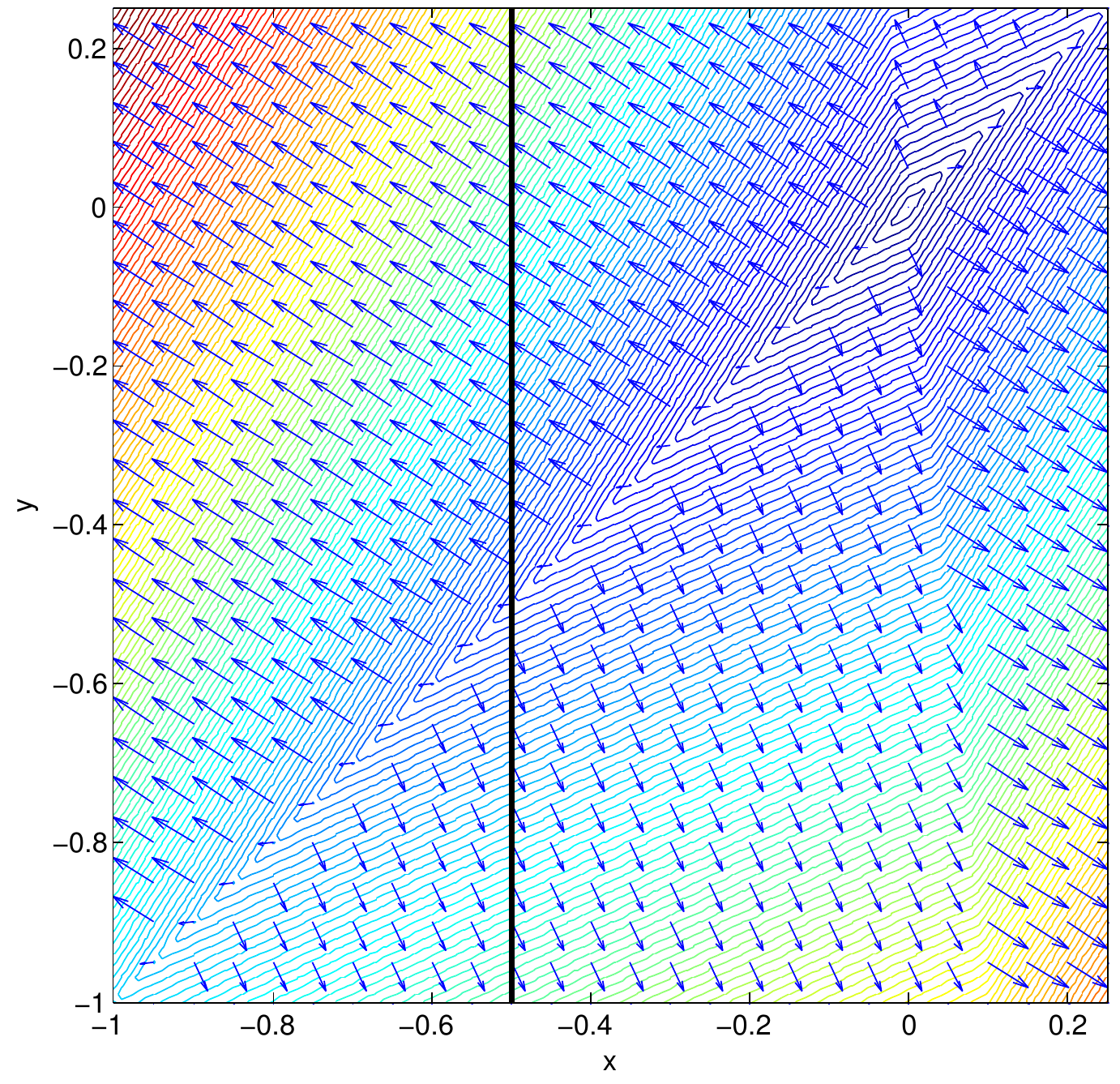}
    \caption{The contour lines of the objective function in \eqref{eq:myopic}.  The arrows indicate the gradients of the function.}
    \label{fig:myopism}
\end{figure}

Consider the problem
\begin{align}
    \min_{x\in \mathbb{R}^2} \  &  |x_1 - x_2| + \frac{1}{2}(x_1 + 0.1 x_2)^2\label{eq:myopic}
\\
    \text{s.t.} \ & x_1 \leq -0.5\nonumber
\end{align}
with optimal solution $x^*=(-0.5,-0.5)^T$. Suppose we have an iterate $x^k=(-0.5,a)^T$ for a number $a\in (-5,-0.5)$. The function is differentiable at this point with gradient $\nabla f(x^k) = (1,-1)^T+(-0.5+0.1a) (1, 0.1)^T=(0.5+0.1a, -1.05+0.01a)^T$. Given that $a>-5$, we have $\nabla_1 f(x^k) >0$. Therefore, the choice $\AAA_k=\BBB(\nabla f(x^k))$  predicts $x_1$ to be free, no matter how close $x^k$ is to $x^*$. This determination is incorrect since $x_1$ is indeed at its bound at the solution.
Notice that the failure of identification is caused by the inherent shortsightedness of the gradient and not due to some kind of degeneracy. 

We point out that also the active-set identification of the L-BFGS-B method does not recognize that $x_1$ is tight at the solution. The L-BFGS-B method, as described previously, locates the first minimizer of the gradient-projection path, $\alpha\mapsto P(x^k - \alpha \nabla f(x^k))$, for the active-set identification. For problem~\eqref{eq:myopic}, the ray $\{x^k - \alpha \nabla f(x^k) : \alpha > 0\}$, for an iterate $x^k=(-0.5,a)^T$ with $a\in(-5,-0.5)$, never intersects a bound. Thus, also in this case, $x_1$ is not recognized as active. 

In order to compensate for the shortsightedness of the gradient, we propose two strategies: (i) an active-set prediction that considers an approximation $\tilde g^k$ of minimum-norm subgradients of nearby non-differentiable points to determine the binding constraints $\BBB(\tilde g^k$) (Section \ref{sec:altactiveset}); and (ii) a correction mechanism that augments the active set if the search direction, computed with L-BFGS approximation of the nonsmooth objective, indicates that a variable should be active (Section \ref{sec:correction}). 




\subsection{Active-Set Prediction Using a Subgradient Approximation}
\label{sec:altactiveset}
Continuing with problem \eqref{eq:myopic}, let us consider the scenario in which we use the $\epsilon$-minimum-norm subgradient ($\epsilon$-MNSG) based on the Clarke $\epsilon$-subdifferential \cite{burke2005robust} for the active-set prediction. For a fixed $\epsilon>0$, the $\epsilon$-MNSG is defined as
\begin{align}
\hat{g}_\epsilon(x) &= \argmin_{y \in \text{cl} \text{ conv } \nabla f({B}_\epsilon(x))} \quad \frac{1}{2} \| y \|_2^2.
\end{align}
Here, $\nabla f({B}_\epsilon(x)) = \{\nabla f(y): y \in B_\epsilon(x)\}$, and the term ``cl conv'' indicates the closure of the convex hull of a set. 
%
When $\epsilon$ is sufficiently small and $a$ is close to $-0.5$, the $\epsilon$-MNSG at $x^k=(-0.5,a)^T$ is approximately $\hat{g}_\epsilon(x^k) \approx (-0.3025,-0.3025)^T$.
The active set $\AAA^k=\mathcal{B}(\hat{g}_\epsilon(x^k))$  correctly identifies that $x_1$ is tight at the solution. Indeed, the $\epsilon$-MNSG attempts to be less myopic than the gradient and forms the basis for gradient-sampling methods \cite{burke2005robust}.
%
This motivates us to base the active-set identification on the $\epsilon$-MNSG instead of the gradient.

Computing the true $\epsilon$-MNSG is usually not feasible, due to the complex nature of $\nabla f(B_\epsilon(x^k))$.  Instead, gradient sampling methods work with an approximation $\tilde g^k$ that is based on the gradients at points from a finite random subsample $G^k=\{x^{k,1},\ldots,x^{k,l^k}\}$ of the ball $B_\epsilon(x^k)$.  More specifically, $\tilde{g}^k= \sum_{i=1}^{l^k} \nabla f(x^{k,i})\lambda^\star_i$ where $\lambda^\star$ is the solution of the convex quadratic problem 
\begin{align}
    \min_{\lambda\in\RR^{l^k}} \; &\frac{1}{2} \left\| \sum_{i=1}^{l^k} \nabla f(x^{k,i}) \lambda_i \right \|_2^2 \nonumber\\
    \text{s.t.} \ & \sum_{i=1}^{l^k} \lambda_i = 1 \label{eqn:msng}\\
    \ &  0\leq \lambda_i \leq 1 \quad \text{ for all } i=1,\ldots,l^k.\nonumber
\end{align}

A good approximation of the $\epsilon$-MNSG typically requires a large number $l^k$ of gradient evaluations. For the purpose of determining the active set, however, an inexact estimate might suffice, since its main purpose is to capture roughly the geometry of the nonsmooth function.  It is not used for the step computation itself.
To avoid additional gradient evaluations, we simply choose $G^k=\{x^k, \ldots, x^{\max\{0,k-M\}}\}$ to contain the most recent $M$ iterates.  This strategy is motivated by two observations: (i) as we will describe in Section \ref{sec:linesearch}, the line search encourages the iterates to cross over manifolds of nondifferentiability and thus, the gradients for $G^k$  represent different ``pieces'' of the nonsmooth function; and (ii) near the solution, where active-set prediction strategies are arguably more important, the steps taken by the algorithm are small and the points in $G^k$ are then from a small neighborhood around the current iterate.



Motivated by these observations, our first active-set selection strategy chooses
\begin{equation}\label{eq:activesetprediction}
\mathcal A^k = \mathcal B^k(\tilde g^k) \cup \mathcal B^k(\nabla f(x^k)).
\end{equation}
Note that we include the bound constraints identified by the gradient $\nabla f(x^k)$ as well.  We observed in our experiments that this led to better performance.  We speculate that, in regions not close to a manifold on which the function is nonsmooth, the active set identified by the gradient is often reliable and the subgradient approximation might cause spurious identification, when the points in $G^k$ are not close to each other. 



\subsection{Computation of the Search Direction}
\label{sec:stepcomputation}

Our second active-set strategy loops over candidate choices for the active set that are evaluated based on the search directions they generate.  We describe the step computation first.

The search directions are based on the BFGS method \cite{mybook}. This method constructs and updates a  convex second-order model of the objective function requiring only the first-order derivatives. Given an estimate, $B^l$, of the curvature of the objective function, the BFGS method revises the estimate using a rank-2 update as
\begin{equation}\label{eq:bfgs}
B^{l+1} = B^l + \frac{{y}^l ({y}^l)^{T}}{({y}^l)^{T} {s}^l} - \frac{B^l {s}^l ({s}^l)^{T} B^l }{({s}^l)^{T} B^l{s}^l},
\end{equation}
where
\begin{subequations}\label{eq:s_and_y}
\begin{align}
s^l &= x^{l+1} - x^l \\ y^l &= \nabla f(x^{l+1}) - \nabla f(x^l).
\end{align}
\end{subequations}
One usually requires that the curvature condition
\begin{equation}
  \label{eq:sycurv}
  (s^l)^Ty^l > 0
\end{equation}
holds, since then $B^{l}$ remains positive definite if the initial matrix $B^0$ is positive definite.
For smooth convex objective functions, the BFGS method possesses strong theoretical properties including global convergence and superlinear local convergence. Even though only limited theoretical convergence guarantees have been established for nonconvex  objectives, many have noted good performance on a variety of problems.

The original BFGS method requires the storage and manipulation of an $n\times n$ matrix. For large-scale problems, this is unwieldy. The limited-memory BFGS (L-BFGS) method \cite{liu1989limited} attempts to alleviate this handicap by storing only the past $m$ curvature pairs $(s^l,y^l)$. The matrix $B^l$ itself is never explicitly constructed.  The value for $m$, also called as the L-BFGS memory, is often in the range of $5$--$20$.  This reduces the storage from $\mathcal{O}(n^2)$ to $\mathcal{O}(mn)$. The complexity of the search direction computation in the L-BFGS method is also reduced from $\mathcal{O}(n^2)$ to $\mathcal{O}(mn)$. 

Given an active set $\mathcal{A}^k$ and an L-BFGS approximation $B^k$ to the curvature of the problem, we generate the search direction $p^k$ as the solution to subproblem \eqref{eq:subprob}.
Here, all components of $p^k$ belonging to $\AAA^k$ are set to zero, and the remaining entries are obtained from a linear system involving a symmetric submatrix of $B^k$.  
Making use of the L-BFGS compact representation matrices \cite{byrd1995limited, byrd1994representations}, the step can be compute efficiently, using $2m^2t+6mt+4t+\mathcal{O}(m^3)$ operations where $t=|\mathcal{A}^k|$. 




To specify the L-BFGS approximation in a given iteration $k$, it is necessary to provide an initial matrix $B^{k,0}$, from which $B^k=B^{k,m}$ is generated by repeatedly applying the update formula \eqref{eq:bfgs}.  The matrix $B^{k,0}$ is an estimate of the curvature of $f$. 
This choice, especially when the L-BFGS memory $m$ is low, has direct consequences on the quality of the search direction. For smooth optimization, $\theta=\frac{(s^k)^Ty^k}{(y^k)^Ty^k}$ is often recommended and is found to work well in a variety of applications. Intuitively, the ratio is justified since it is a scalar approximation to $\nabla^2 f(x^k)$ \cite{mybook}. However, for nonsmooth optimization this choice seems to lead to inferior performance. Instead, Curtis and Que \cite{curtis2015quasi} proposed $\theta=\max(1.0,\min(\|\nabla f(x^k)\|_\infty,10^8))$. We use this choice in our implementation as well.



\subsection{An Active-Set Correction Mechanism}
\label{sec:correction}
In Section \ref{sec:altactiveset}, we described an active-set identification mechanism that is based on an approximation of a subgradient. This strategy attempts to {guess} directly which bounds are tight at the optimal solution.

Next we describe another approach using an iterative correction procedure that judges the quality of a candidate active set and adjusts it if necessary. The quality of the active set is adjudged through the search direction generated using it. 
The goal is to obtain a direction that is feasible in the sense that a sufficiently small step into this direction does not leave the feasible region.
%
If, for a given candidate active set, there is a variable that is tight at the current iterate and the candidate search direction points outside the feasible region, then this variable is added to the active set and the procedure is repeated.
Similar mechanisms have been used previously, for example, for solving convex quadratic programs \cite{CurtisHanRobinson2015,hungerlander2015feasible} and $\ell_1$-regularized convex optimization problems \cite{Byrd2016,obapaper}.

For ease of notation, we drop the iteration index $k$ for the remainder of this section. Given an iterate $x$, we let $g=\nabla f(x)$, and we define the set of interior variables, $\FFF=\{i: l_i < x_i < u_i \}$ and the set of variables with tight bound constraints, $\bFFF=\NNN\setminus\FFF$.
Algorithm~\ref{alg:correction} formally states the proposed active-set correction procedure.
When there are no tight bounds, the active set must be empty, and step~\ref{alg:nothing} immediately returns the unrestricted search direction in the full space.
Otherwise, the $t$-loop computes a potential search direction in step~\ref{alg:subprob} and tests if there are any components, collected in the set $\CCC^t$, that would take a variable instantaneously outside its bound constraints.
Such components are added to the active set, until a feasible direction is found. 
Clearly, the loop terminates in finite time, since $\AAA^{t}$ grows by at least one element per iteration.


\begin{algorithm}[t]
\caption{ActiveSetCorrection}
\label{alg:correction}
{\bf Inputs:} Current iterate $x$; initial active set $\AAA^{\textup{init}}\subseteq\bFFF$.

{\bf Output:} Final active set $\AAA$ with corresponding search direction $p$.
\begin{algorithmic}[1]
\State\label{alg:initcor} Initialize $t\gets 0$ and set $\AAA^0 = \AAA^{\textup{init}}$.
\If{$\bFFF=\emptyset$} \Comment{Nothing to do if there are no tight bounds}
\State Compute $p$ as solution of \eqref{eq:subprob} with $\AAA^k=\emptyset$.
\State \Return{$\AAA=\emptyset$ and $p$.}\label{alg:nothing}
\EndIf
\For {$t=0,1,2,\ldots$}
\State\label{alg:stepcor} Compute $p^t$ as solution of \eqref{eq:subprob} with $\AAA^k=\AAA^t$. \Comment{Potential search direction}\label{alg:subprob}
\State Set $\CCC^{t}=\{i\in\bFFF\setminus\AAA^t : T(x,p^t)_i \neq p^t_i\}$. \Comment Variables to be added\label{alg:defC}
\If{$\CCC^{t}=\emptyset$} 
\State \Return{$\AAA=\AAA^t$ and $p=p^t$}. \Comment{No more corrections necessary}
\EndIf
\State Set $\mathcal{A}^{t+1} = \mathcal{A}^t \cup \mathcal{C}^{t}$.
\EndFor
\end{algorithmic}
\end{algorithm}

The approach described in Section~\ref{sec:altactiveset} uses a subgradient approximation to overcome the shortsightedness of the gradient when predicting the optimal active set.  In contrast, the corrective procedure in Algorithm~\ref{alg:correction} exploits the fact that the L-BFGS approximation of the problem curvature contains information about the structure of the nonsmoothness.  
It has been observed that the (L-)BFGS approximation of a nonsmooth objective function is able to approximate the U- and V-spaces of the objective function \cite{Lewis2013, skajaa2010limited}.  Roughly speaking, the U-space of $f$ at a point $x$ is the subspace tangent to the manifold of points at which $f$ is not differentiable.  The V-space is the orthogonal complement of the U-space.   In \cite{Lewis2013}, Lewis and Overton hypothesized that the V-space of a nonsmooth function can be numerically approximated within the unadulterated BFGS method through the eigenvectors of $B^k$ corresponding to eigenvalues that converge to infinity.   

In our example problem \eqref{eq:myopic}, the V-space at any point $x$ with $x_1=x_2$, including the optimal solution, is spanned by $(1,-1)^T$.
When we apply the proposed method from random starting points, the iterates converge to the solution $(-0.5, -0.5)^T$. After some iterations, the BFGS matrix is approximately
\begin{equation}\label{eq:Bk_approx}
B^k \approx y^k \cdot \begin{bmatrix} 1 & -1 \\ -1 & 1\end{bmatrix}
\end{equation}
with some sequence $y^k$ converging to infinity. Note that the eigenvectors of $B^k$ suggest precise recovery of the U- and V-spaces of the objective function. In particular, the eigenvector $(1,-1)^T$ with respect to the asymptotically infinite eigenvalue indeed spans the V-space of $f$ at the optimal solution.
Even though the matrix on the right-hand side of \eqref{eq:Bk_approx} is singular, $B^k$ itself is always nonsingular.  Numerically we observe that the inverse matrix $H^k=(B^k)^{-1}$ is approximately
\begin{equation}\label{eq:Hk_approx}
H^k \approx \tilde y^k\cdot \begin{bmatrix} 1 & 1 \\ 1 & 1\end{bmatrix}
\end{equation}
with some sequence $\tilde y^k$ converging to zero, and now $(1,-1)^T$ is an eigenvector with respect to the eigenvalue approaching zero.

Dropping the iteration index $k$, consider again an iterate of the form $x=(-0.5,a)^T$ with $a\in (-5,-0.5)$ and gradient $\nabla f(x) = (0.5+0.1a, -1.05+0.01a)^T$.  In Section~\ref{sec:activeset} we observed that the na{\"i}ve choice $\AAA=\mathcal B(\nabla f(x))=\emptyset$ fails to recognize that $x_1$ is active at the solution.  If we choose $\AAA^{\textup{init}}=\mathcal B(\nabla f(x))$ in Algorithm~\ref{alg:correction}, we have $\AAA^0=\emptyset$ in the first iteration.  With the approximation \eqref{eq:Hk_approx}, the search direction in step \ref{alg:stepcor} becomes $p^0\approx\tilde y\cdot (0.55-0.11a, 0.55-0.11a)^T$.  Clearly, from the current iterate with $x_1=-0.5$, this direction points out of the feasible region because $p^0_1>1$.  Therefore, $T(x, p^0)_1 \neq p^0_1$ and $\mathcal C^0=\{1\}$.  In the next iteration of the correction loop, $\AAA^1=\{1\}$ is accepted as the final active set, and correctly predicts that $x_1$ is active at the solution.

The following lemma shows that the correction mechanism cannot lead to a spurious termination of the overall algorithm.  A zero step can be generated only when the current iterate is already a stationary point of the objective function (assuming that $\AAA^0$ is initialized as $\BBB(\nabla f(x^k))$).

\begin{lemma}\label{lem:p0}
Suppose $\AAA^0=\BBB(\nabla f(x^k))$ and the corrective loop terminates with $p=0$. Then, the current iterate $x^k$ satisfies the first order optimality conditions \eqref{eq:firstorder} for problem \eqref{eq:problem}.
\end{lemma}

\begin{proof}[Proof of Lemma~\ref{lem:p0}]

For ease of exposition, we assume that $l=0$ and $u=\infty$ in problem \eqref{eq:problem}. 
Let $\hat t$ be the iteration in which the method terminates with $p=p^{\hat t}=0$ and let $g=\nabla f(x)$.

Given the active set $\AAA^{\hat t}\subseteq\bFFF$ and defining $\bAAA^{\hat t}:=\bFFF\setminus\AAA^{\hat t}$, the solution $p$ to \eqref{eq:subprob} (with $\AAA^k=\AAA^{\hat t}$) is obtained by solving the linear system
\begin{equation}\label{eq:pF}
\begin{bmatrix}
   B_{\FFF\FFF} & B_{\FFF\bAAA^{\hat t}} \\ B_{\bAAA^{\hat t}\FFF} & B_{\bAAA^{\hat t}\bAAA^{\hat t}}
\end{bmatrix}
\begin{pmatrix}p_{\FFF} \\ p_{\bAAA^{\hat t}}\end{pmatrix} = - 
\begin{pmatrix}
  g_{\FFF} \\ g_{\bAAA^{\hat t}}
\end{pmatrix}
\end{equation}
and setting
\begin{equation}\label{eq:pA}
p_{\AAA^{\hat t}}=0.
\end{equation}
%
 Because the L-BFGS approximation $B$ is positive definite, \eqref{eq:pF} together with $p^{\hat t}=0$ implies 
\begin{equation}\label{eq:g_zero}
g_{\FFF}=0 \text{ and } g_{\bAAA^{\hat t}}=0.  
\end{equation}

For the purpose of deriving a contradiction suppose that $\hat t>0$. Then $\AAA^{\hat t} = \AAA^{\hat t-1} \cup \CCC^{\hat t-1}$, and therefore $\bAAA^{\hat t-1} = \bAAA^{\hat t} \cup \CCC^{\hat t-1}$.  By definition, $p^{\hat t-1}_i<0$ for every $i\in\CCC^{\hat t-1}$.  Since $\CCC^{\hat t-1}\neq\emptyset$, we have $p^{\hat t-1}\neq0$.
Consider the linear system from which $p^{\hat t-1}$ is computed:
\[
  \underbrace{
    \begin{bmatrix}
    B_{\FFF\FFF} & B_{\FFF\bAAA^{\hat t}} & B_{\FFF\CCC^{\hat t-1}} \\ 
    B_{\bAAA^{\hat t}\FFF} & B_{\bAAA^{\hat t}\bAAA^{\hat t}} & B_{\bAAA^{\hat t}\CCC^{\hat t-1}} \\ 
    B_{\CCC^{\hat t-1}\FFF} & B_{\CCC^{\hat t-1}\bAAA^{\hat t}} & B_{\CCC^{\hat t-1}\CCC^{\hat t-1}} \\ 
  \end{bmatrix}
  }_{=: \hat B^{\hat t-1}}
  \underbrace{\begin{pmatrix}p^{\hat t-1}_{\FFF} \\ p^{\hat t-1}_{\bAAA^{\hat t}}\\ p^{\hat t-1}_{\CCC^{\hat t-1}}\end{pmatrix}}_{=:\hat p^{\hat t-1}} = - 
  \underbrace{\begin{pmatrix}
    g_{\FFF} \\ g_{\bAAA^{\hat t}} \\ g_{\CCC^{\hat t-1}}
  \end{pmatrix}}_{=:\hat g^{\hat t-1}}.
\]
With \eqref{eq:g_zero} we obtain 
\begin{equation}\label{eq:gTp}
(g_{\CCC^{\hat t-1}})^Tp^{\hat t-1}_{\CCC^{\hat t-1}} = (g^{\hat t-1}) ^T \hat p^{\hat t-1}= - (\hat p^{\hat t-1})^TB^{\hat t-1}\hat p^{\hat t-1}<0
\end{equation}
because $B$ is positive definite and $p^{\hat t-1}\neq0$.

On the other hand, $\CCC^{\hat t-1}\subseteq \bFFF\setminus \AAA^{\hat t-1}\subseteq \bFFF\setminus \AAA^{\hat t-2}\subseteq \ldots \subseteq \bFFF\setminus \AAA^{0} = \bFFF\setminus \BBB(g)$.  From \eqref{eq:activeset} we then have $g_i<0$ for all $i\in\CCC^{\hat t-1}$.  Also, from the definition of $\CCC^{\hat t-1}$, it is $p^{\hat t-1}_i<0$ for all $i\in\CCC^{\hat t-1}$. Therefore, 
\[
(g_{\CCC^{\hat t-1}})^Tp^{\hat t-1}_{\CCC^{\hat t-1}} = \sum_{i\in\CCC^{\hat t-1}}g_ip_i^{\hat t-1} >0,
\]
in contradiction to \eqref{eq:gTp}.

It follows that $\hat t$ must be zero, and \eqref{eq:g_zero} yields that $g_i=0$ for any $i\not\in\AAA^0=\BBB(\nabla f(x^k))$.  Consequently, \eqref{eq:firstorder} holds.
\end{proof}



\subsection{Line search}
\label{sec:linesearch}

Once a search direction $p^k$ at an iterate $x^k$ has been calculated, the algorithm determines a step size $\alpha^k>0$ to generate the next iterate, $x^{k+1} = P(x^k+\alpha^kp^k)$. The projection ensures that the new iterate is feasible.

For the unconstrained minimization of a nonsmooth function, Lewis and Overton \cite{Lewis2013} use the weak Wolfe conditions \eqref{eq:uncon_armijo} and \eqref{eq:uncon_curvature} to determine whether a trial point $x^{\textup{trial}}=x^k+\alpha p^k$ is acceptable as a new iterate.  Given fixed values for $c_1, c_2\in(0,1)$ with $c_1<c_2$, the first condition,
\begin{equation}\label{eq:uncon_armijo}
    f(x^{\textup{trial}}) \leq f(x^k) + \alpha c_1 \nabla f(x^k)^Tp^k,
\end{equation}
ensures that the objective function decreases by at least a fraction of what is predicted by a linear approximation that is based on the gradient.  Because the objective is nonsmooth, the linear model might be a good approximation only for very small step sizes $\alpha$.  Nevertheless, since $f$ is assumed to be differentiable at $x^k$, condition \eqref{eq:uncon_armijo} can be satisfied as long as $\alpha$ is sufficiently small.

The second condition,
\begin{equation}\label{eq:uncon_curvature}
    \nabla f(x^{\textup{trial}})^T p^k \geq c_2 \nabla f(x^k)^Tp^k,
\end{equation}
imposes that the slope of the function $\phi(\alpha)=f(x^k+\alpha p^k)$ is less steep at $\alpha$ than at $0$, indicating that sufficient progress towards a local minimizer of $\phi(\alpha)$ is made.  With a nonsmooth objective, a local minimizer of $\phi(\alpha)$ might be at a point where $f$ (and hence $\phi$) is nondifferentiable.  If the current iterate is close to such a point, requiring \eqref{eq:uncon_curvature} leads to a step size $\alpha$ that is beyond the point of nondifferentiability.  This observation is crucial and provides the main intuition why the BFGS algorithm works well for nonsmooth problems.
Because the next iterate lies on another ``smooth piece'' of the nonsmooth function, the new gradient is quite different from the current gradient, even when the next iterate is very nearby.  Recalling the definition \eqref{eq:s_and_y} of $s^k$ and $y^k$ in the BFGS update, we see that then the gradient difference $y^k$ is much larger in size than the step $s^k$. Consequently, high curvature in the direction $s^k$ is incorporated into the BFGS update $B^{k+1}$, approximating the infinite curvature at the point of nondifferentiability.

Lewis and Overton \cite{Lewis2013} prove that, in the absence of bounds, a step size $\alpha$ satisfying both \eqref{eq:uncon_armijo} and \eqref{eq:uncon_curvature} always exists, and they provide a bracketing procedure to find it.  We point out that \eqref{eq:uncon_curvature} also guarantees that the $(s^k,y^k)$ pair for the BFGS update satisfies $(s^k)^Ty^k>0$ so that the update is well-defined.

On the other hand, for the minimization of a \emph{smooth} objective function subject to bound constraints, Ferry \cite{ferry2011projected} proposed a generalized Wolfe line search which replaces the search direction by $\bar p^k=T(x^k,p^k)$.  Recalling the definition of $T$ in \eqref{eq:def_T}, we see that $\bar p^k$ is the modified search direction that zeros out all components that would result in an immediate violation of a constraint.
With this, the Wolfe conditions suggested by Ferry \cite{ferry2011projected} are
\begin{align}\label{eq:con_armijo}
    f(x^{\textup{trial}}) &\leq f(x^k) + \alpha c_1 \nabla f(x^k)^T\bar p^k \\
    \nabla f(x^{\textup{trial}})^T T(x^{\textup{trial}}, p^k)  & \geq c_2 \nabla f(x^k)^T\bar p^k. \label{eq:con_curvature}
\end{align}
We adopt these conditions in our context of minimizing a nonsmooth objective.  Algorithm~\ref{alg:generalized-wolfe} describes the corresponding bracketing mechanism.
\begin{algorithm}[t]
\caption{ModifiedWolfe}
\label{alg:generalized-wolfe}
{\bf Inputs:} Current iterate $x$ and a search direction $p$.

{\bf Output:} Step size $\alpha$ to generate next iterate.

{\bf Parameters:} $c_1, c_2\in (0,1)$ with $c_1<c_2$; $\epsilon_{\textup{abs}},\epsilon_{\textup{rel}} > 0$.
\begin{algorithmic}[1]
\State Set $L = 0$
\State Set $U = \max_i \{ \gamma_i \}$, where $\gamma_i \defeq \begin{cases}
\frac{u_i-x_i}{p_i} & p_i>0 \text{ and } x_i \neq u_i;\\
\frac{x_i-l_i}{p_i} & p_i<0 \text{ and } x_i \neq l_i;\\
\infty & \text{otherwise.}
\end{cases}$
\State Set $\alpha = \min(1,U)$.
\State Compute $\bar p=T(x,p)$.
\If{$\bar p=0$}
\State \textbf{terminate} with ``No search direction''.\label{alg:zero_step} 
\Comment Error due to bad search direction
\EndIf
\For{$t=0,1,2,\ldots$}
\State Set $x^{\textup{trial}} = P(x + \alpha \bar p)$.
    \If{\eqref{eq:con_armijo} does not hold} \Comment{Sufficient decrease condition does not hold}
        \State Set $U = \alpha$.\label{alg:U}
    \Else
        \If{\eqref{eq:con_curvature} does not hold} \Comment{Curvature condition does not hold}
           \State  Set $L = \alpha$.\label{alg:L}
        \Else
            \State \Return{$\alpha$}. \Comment{Return step satisfying weak Wolfe conditions}\label{alg:success}
        \EndIf
    \EndIf
\If{$U<\max_i \{ \gamma_i \}$} \Comment{Update step-size}
\State Set $\alpha = \frac{U+L}{2}$.\label{alg:update1}
\Else
\State Set $\alpha = \min(2L,U)$.\label{alg:update2}
\EndIf
\If{$U-L<\epsilon_{\textup{abs}}+\epsilon_{\textup{rel}}L$ } \label{alg:term}
\If{$L>0$}
\State \Return{L}.\label{alg:returnL} \Comment{Return step satisfying sufficient decrease condition}
\Else
\State \textbf{terminate} with ``Line Search Error''.\label{alg:error}\Comment{Error, no suitable step size found}
\EndIf
\EndIf
\EndFor
\end{algorithmic}
\end{algorithm}
Note that, in the absence of bounds, this procedure is identical to the line search algorithm proposed by Lewis and Overton \cite{Lewis2013}.
When $f$ is smooth, Ferry \cite{ferry2011projected} showed that there always exists a step size $\alpha$ that satisfies both \eqref{eq:con_armijo} and \eqref{eq:con_curvature}.
For a nonsmooth objective, such a step size may not exist.  In our method, when a suitable step size cannot be found in step \ref{alg:success}, finite termination is ensured by the termination test in step \ref{alg:term}.

The bracketing mechanism generates a sequence of values for $U$ and $L$ in a way that shrinks the length of the interval $[L,U]$ to zero (see steps \ref{alg:update1} and \ref{alg:update2} together with steps \ref{alg:U} and \ref{alg:L}).  Because it is not clear whether a step size satisfying both \eqref{eq:con_armijo} and \eqref{eq:con_curvature} can be found, the algorithm will attempt only a moderate number of trial step sizes, until the relative interval length is on the order of $\epsilon_{\textup{rel}}$.
Whenever a step size is encountered that satisfies the sufficient decrease condition \eqref{eq:con_armijo}, step~\ref{alg:L} sets $L$ to this value (unless the search is terminated in step~\ref{alg:success}).  Therefore, when step~\ref{alg:returnL} returns a nonzero step size, it is guaranteed that the next iterate will have a smaller objective value, and so the overall optimization algorithm cannot cycle.
On the other hand, if $L$ is zero in step~\ref{alg:error}, the trial step size $U=\alpha$ has become smaller than $\epsilon_{\textup{abs}}$.  In that case, we declare a line search error, which is likely caused by numerical issues in the search direction computation or round-off in the function evaluation.
Finally, it may happen that the computed search direction is such that $P(x+\alpha p)=x$ for any $\alpha>0$.  Then there is no point in conducting a line search and we terminate the optimization with an error message in step~\ref{alg:zero_step}.  This may occur due to 
numerical problems during the step computation.

\subsection{Main Algorithm}
\label{sec:mainalgo}

\begin{algorithm}[t]
\caption{Nonsmooth Quasi-Newton (\algoname)}
\label{alg:nqn}
{\bf Inputs:} Initial point $x^0\in[l,u]$.

{\bf Parameters:} Size of L-BFGS memory $m$; update tolerance $\epsilon_{\textup{skip}}$.
\begin{algorithmic}[1]
\State Initialize storage $\mathcal S$ of L-BFGS curvature pairs to be empty.
\For {$k=0,1,2,\ldots$}
\If{$T(x^k,-\nabla f(x^k)=0$}
\State \Return $x^k$ \label{alg:firstorder} \Comment{Finite termination at stationary point}
\EndIf
\State Choose active set $\mathcal{A}^k$. \label{alg:active_and_step}\Comment{Details specified elsewhere}

\State Compute search direction $p^k$.\label{alg:pk}
\State Compute $\alpha^k=$ \texttt{ModifiedWolfe}$(x^k,p^k)$. \Comment{Perform line search using Algorithm 2}

\State Set $\bar p^k = T(x^k,p^k)$.
\State Set $x^{k+1} = {P}(x^k + \alpha^k \bar p^k)$.\label{alg:proj}

\State Compute curvature pair $(s^k,y^k)$ from \eqref{eq:s_and_y}.

\If{$(s^k)^T y^k > \epsilon_{\textup{skip}} \|s^k\| \|y^k\|  $} \label{alg:skip} \Comment{Discard pair if curvature condition not satisfied}
\State Store $(s^k,y^k)$ in $\mathcal S$.\Comment Update L-BFGS memory
\State If $|\mathcal S|>m$ then discard oldest curvature pair.
\EndIf
\EndFor

\end{algorithmic}
\end{algorithm}

The overall optimization algorithm for solving \eqref{eq:problem} is given in 
Algorithm \ref{alg:nqn}.  Step~\ref{alg:active_and_step} is purposely left vague, because we will explore different alternatives for the active-set selection.  The experiments in the following section consider the following options:
\begin{description}
\item[Variant 1] Choose the active set based on the gradient at the current iterate, $\AAA^k=\BBB(\nabla f(x^k))$.
\item[Variant 2] Choose the active set based on the subgradient approximation  using \eqref{eq:activesetprediction}.
\item[Variant 3] Compute the active set from the correction procedure Algorithm~\ref{alg:correction} with initial guess $\AAA^{\textup{init}} = \BBB(\nabla f(x^k))$. 
\item[Variant 4] Compute the active set from the correction procedure Algorithm~\ref{alg:correction} with initial guess $\AAA^{\textup{init}}$ based on the subgradient approximation  using \eqref{eq:activesetprediction}.
\end{description}
For the last two variants, the search direction is already computed as byproduct of the active set selection and step~\ref{alg:pk} does not require any actual work.



There is no guarantee that the pair $(s^k,y^k)$ pair defined in \eqref{eq:s_and_y} satisfies the curvature condition \eqref{eq:sycurv}, even when the weak Wolfe conditions \eqref{eq:con_armijo} and \eqref{eq:con_curvature} are satisfied, since the actual step $s^k$ might be different from $\alpha^k\bar p^k$, due to the projection in step~\ref{alg:proj}. This is in contrast to the unconstrained case, where \eqref{eq:con_curvature} implies that \eqref{eq:s_and_y} holds. 
To handle this situation, the update is skipped in step~\ref{alg:skip} whenever $(s^k)^T(y^k)\leq\epsilon_{\textup{skip}} \|s^k\| \|y^k\|$.



As mentioned in Section 1, we do not consider any theoretical convergence properties of this method, including the possibilities of stalling at non-stationary points and of spurious termination of the line search. We point out that convergence guarantees remain an open question even when no constraints are present.


\section{Numerical Experiments}
\label{sec:numericalexperiments}
\subsection{Implementation and Problem Set}
\label{sec:problemset}
We implemented Algorithm \ref{alg:nqn} in Python.  We will refer to it as \algoname.  The code for our algorithm can be found in our GitHub repository: {\url{https://github.com/keskarnitish/NQN}}. The values for the various parameters used are summarized in Table \ref{tab:parameters}. In order to solve the quadratic program \eqref{eqn:msng} for the subgradient approximation, we used the CVXOPT package \cite{andersen2013cvxopt}. We rely on the NumPy package \cite{numpy} for linear algebra operations, and Theano \cite{team2016theano} is used to compute derivatives of the objective functions using algorithmic differentiation.

\begin{table}[t]
    \centering
    \begin{tabular}{|c|c|l|}
    \hline
        Parameter & Value & Description \\ \hline \hline
        $m$ & 20 & L-BFGS memory \\ \hline
        $M$ & $20$ & Maximum size of sample set $G^k$ \\ \hline
        $(c_1,c_2)$ & $(10^{-8},0.9)$ & Parameters on Wolfe conditions \\ \hline
        $\epsilon_{\textup{abs}}$ & $10^{-16}$ & Absolute bracketing tolerance\\ \hline
        $\epsilon_{\textup{rel}}$ & $10^{-6}$ & Relative bracketing tolerance\\ \hline        
        $\epsilon_{\textup{skip}}$ & $10^{-8}$ & Tolerance for skipping L-BFGS update \\  \hline
    \end{tabular}
    \caption{Parameter values used for numerical experiments.}
    \label{tab:parameters}
\end{table}

We compare \algoname with other codes for solving \eqref{eq:problem}, namely (i) L-BFGS-B \cite{byrd1995limited}; (ii) L-BFGS-B-NS \cite{lbfgsbns}; and (iii) LMBM-B \cite{karmitsa2010limited}. While L-BFGS-B is not designed to solve nonsmooth problems, we nonetheless include it owing to its documented success at solving smooth bound-constrained problems. We include L-BFGS-B-NS and LMBM-B since they are specifically designed to solve \eqref{eq:problem} and have shown competitive performance on variety of tasks. The former is identical to the L-BFGS-B algorithm except that it uses the weak Wolfe line search as opposed to the strong Wolfe line search. LMBM-B \cite{karmitsa2010limited} combines the limited-memory bundle method (LMBM) \cite{lmbmnew,haarala2007globally} with a Cauchy-point-based active-set selection strategy similar to the one in L-BFGS-B. The LMBM-B method generates steps using a subgradient bundle in conjunction with an L-BFGS/SR-1 updating scheme to gain curvature information. 

To make sure each solver obtains the same function and derivative information, we implemented Python wrappers around the Fortran codes written by the respective authors. The original Fortran codes can be found at \url{users.iems.northwestern.edu/~nocedal/lbfgsb.html}, \url{github.com/wilmerhenao/L-BFGS-B-NS} and \url{napsu.karmitsa.fi/lmbm/} for L-BFGS-B, L-BFGS-B-NS and LMBM-B respectively. We exclude other methods, including gradient-sampling methods, since we found their performance to be inferior to the methods listed above.


To explore the effect of the different active-set identification mechanisms, we propose two generalizations of \eqref{eq:myopic} as test problems for nonsmooth optimization. The two problems are defined for even values of $n$; we call them \texttt{Myopic\_Decoupled} and \texttt{Myopic\_Coupled}. They are given as
\begin{align}\label{eq:myopic_decoupled}
    \min_{x\in\RR^n} & \sum_{i\in \{1,3,\cdots,n-1\}} |x_i - x_{i+1}| + (x_i + 0.1 x_{i+1})^2,
\end{align} and
\begin{align}\label{eq:myopic_coupled}
    \min_{x\in\RR^n} & \sum_{i=1}^{n-1} |x_i - x_{i+1}| + (x_i + 0.1 x_{i+1})^2,
\end{align} 
respectively; the bound constraints for these problems are discussed below. The attributes \texttt{Decoupled} and \texttt{Coupled} refer to whether or not the problem is separable. 

In addition, we use several test problems from the literature that are listed  in Table \ref{tab:problems} along with their references\footnote{The objective function for problem 20, suggested by Michael Overton \cite{overtonprivate}, is $\max\left\{|x_1|, \max_{i\in \{2,3,\cdots,n\}}|x_{i-1} - x_i| \right\}$.}.
\begin{table}
    \centering
    \begin{tabular}{|r|l|l|}
    \hline
         Problem Number & Test Problem & Reference \\ \hline \hline

        1  & \verb|Active_Faces| & \cite{lmbmnew} \\ \hline
        2  & \verb|Chained_CB3_1| & \cite{lmbmnew} \\ \hline
        3  & \verb|Chained_CB3_2| & \cite{lmbmnew} \\ \hline
        4                       & \verb|Chained_Crescent_1| & \cite{lmbmnew} \\ \hline
        5  & \verb|Chained_Crescent_2| & \cite{lmbmnew} \\ \hline
        6  & \verb|Chained_LQ| & \cite{lmbmnew} \\ \hline
7  & \verb|Chained_Mifflin_2| & \cite{lmbmnew} \\ \hline
8 & \verb|Convex_Nonsmooth| & \cite{skajaa2010limited} \\ \hline
9          & \verb|L1| & \cite{skajaa2010limited} \\ \hline
10 & \verb|L1HILB| & \cite{lmbmnew} \\ \hline
11          & \verb|L2| & \cite{Lewis2013} \\ \hline
12 & \verb|MAXHILB| & \cite{lmbmnew} \\ \hline
13          & \verb|MAXQ| & \cite{lmbmnew} \\ \hline
 14         & \verb|Modified_Rosenbrock_1| & \cite{lbfgsbns} \\ \hline
 15         & \verb|Modified_Rosenbrock_2| & \cite{lbfgsbns} \\ \hline
 16         & \verb|Myopic_Coupled| & \eqref{eq:myopic_coupled} \\ \hline
 17         & \verb|Myopic_Decoupled| & \eqref{eq:myopic_decoupled} \\ \hline
18 & \verb|Nesterov_1| & \cite{gurbuzbalaban2012nesterov} \\ \hline
19          & \verb|Nesterov_2| & \cite{gurbuzbalaban2012nesterov} \\ \hline
20          & \verb|Nesterov_3| & \cite{overtonprivate} \\ \hline
21          & \verb|Nonsmooth_Brown| & \cite{lmbmnew} \\ \hline
22            & \verb|TEST29_2| & \cite{ufo} \\ \hline
23          & \verb|TEST29_6| & \cite{ufo} \\ \hline
24          & \verb|TEST29_22| & \cite{ufo} \\ \hline
25          & \verb|TEST29_24| & \cite{ufo} \\ \hline
    \end{tabular}
    \caption{Test problems used in numerical experiments.}
    \label{tab:problems}
\end{table}
Since these test problems are for unconstrained optimization, we follow an approach similar to \cite{karmitsa2010adaptive} in that we add the following bounds to all problems:
\begin{align*}
    l_i &= \begin{cases}
    [x^\star_{\textup{uncon}}]_i - 5.5 & \text{if } \mod(i,2)=0\\
    -100 & \text{if } \mod(i,2)=1
    \end{cases}\\
    u_i &= \begin{cases}
    [x^\star_{\textup{uncon}}]_i -0.5 & \text{if } \mod(i,2)=0\\
    100  & \text{if } \mod(i,2)=1
    \end{cases}
\end{align*}
where $x^\star_{\textup{uncon}}$ is the unconstrained global minimizer which is known for all problems in closed form. By construction, for all problems, the unconstrained minimizer lies outside the bounds.
For each of the 25 problems, ten starting points were generated randomly via a uniform distribution $U(-2,2)$ centered at the midpoint of the bounds, giving rise to a total of 250 instances.
Each code was run until the number of gradient evaluations exceeded $100n$ or an error occurred.

\subsection{Effect of Active-Set Prediction and Correction Mechanism}
\label{sec:effect-of-activeset}
We begin by investigating the efficacy of the active-set prediction, the correction mechanism, and their interplay, using the four variants given in Section~\ref{sec:mainalgo}.


\begin{figure}[t] 
\begin{subfigure}{0.48\textwidth}
\includegraphics[width=\linewidth]{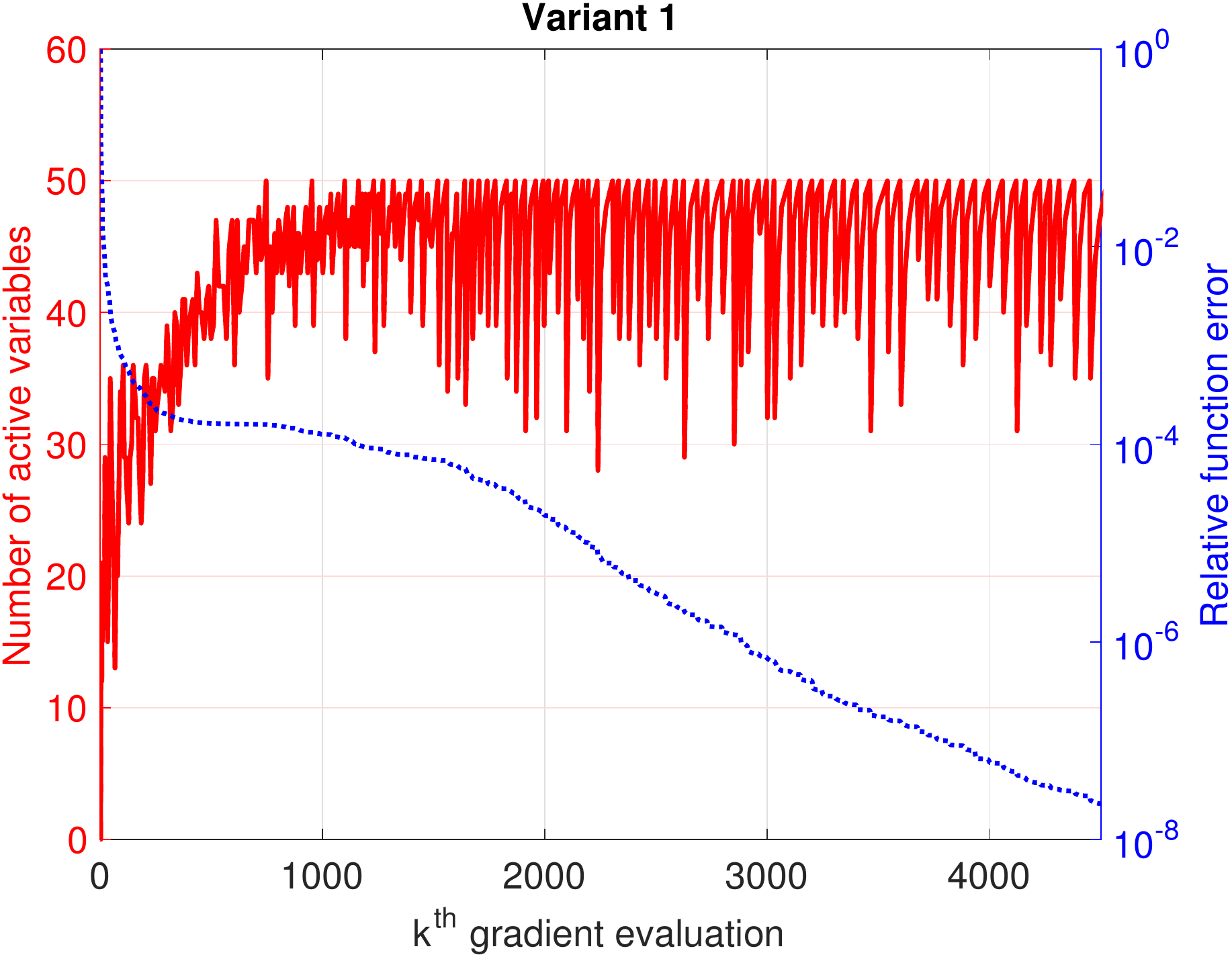}
\end{subfigure}\hspace*{\fill}
\begin{subfigure}{0.48\textwidth}
\includegraphics[width=\linewidth]{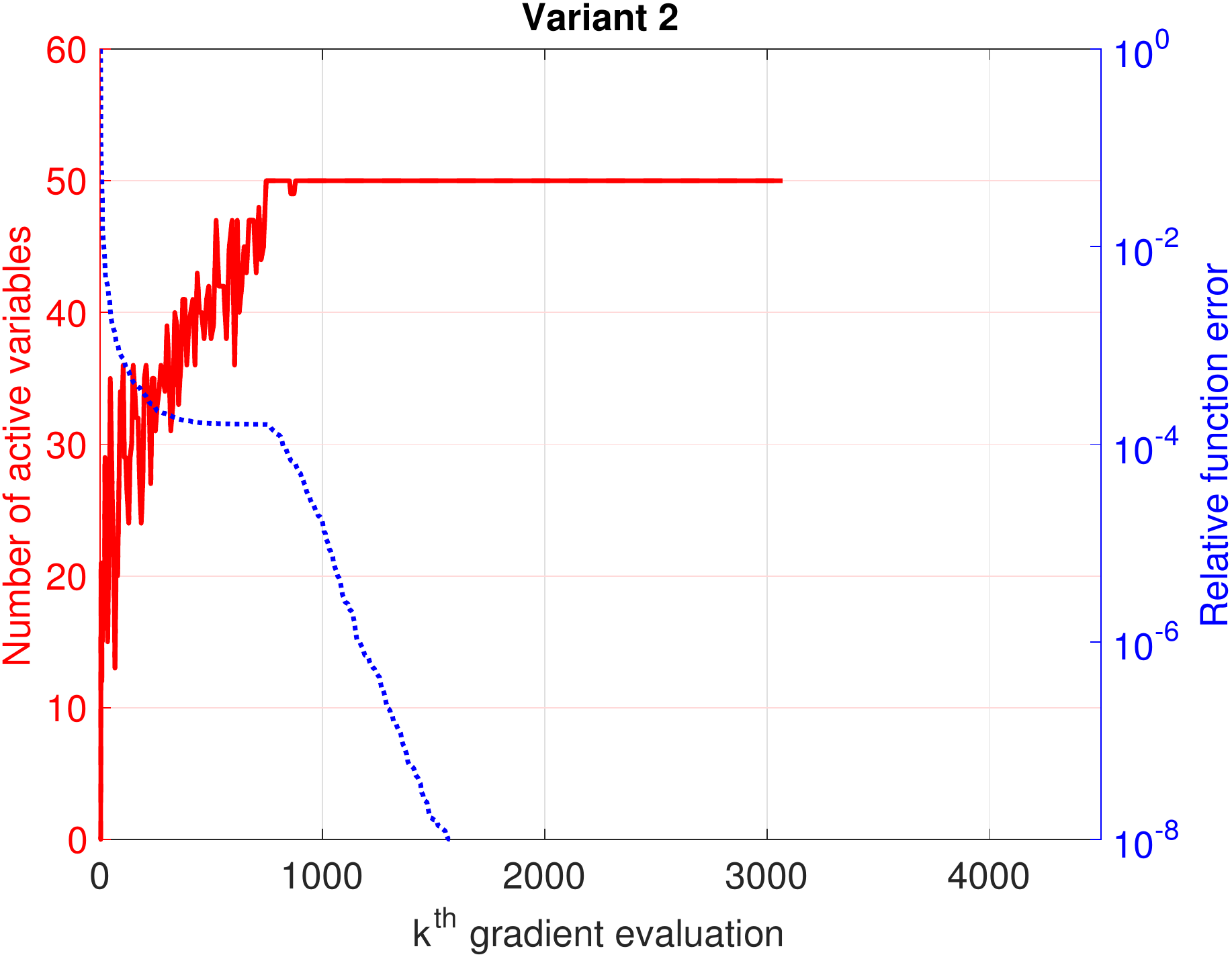}
\end{subfigure}

\medskip
\begin{subfigure}{0.48\textwidth}
\includegraphics[width=\linewidth]{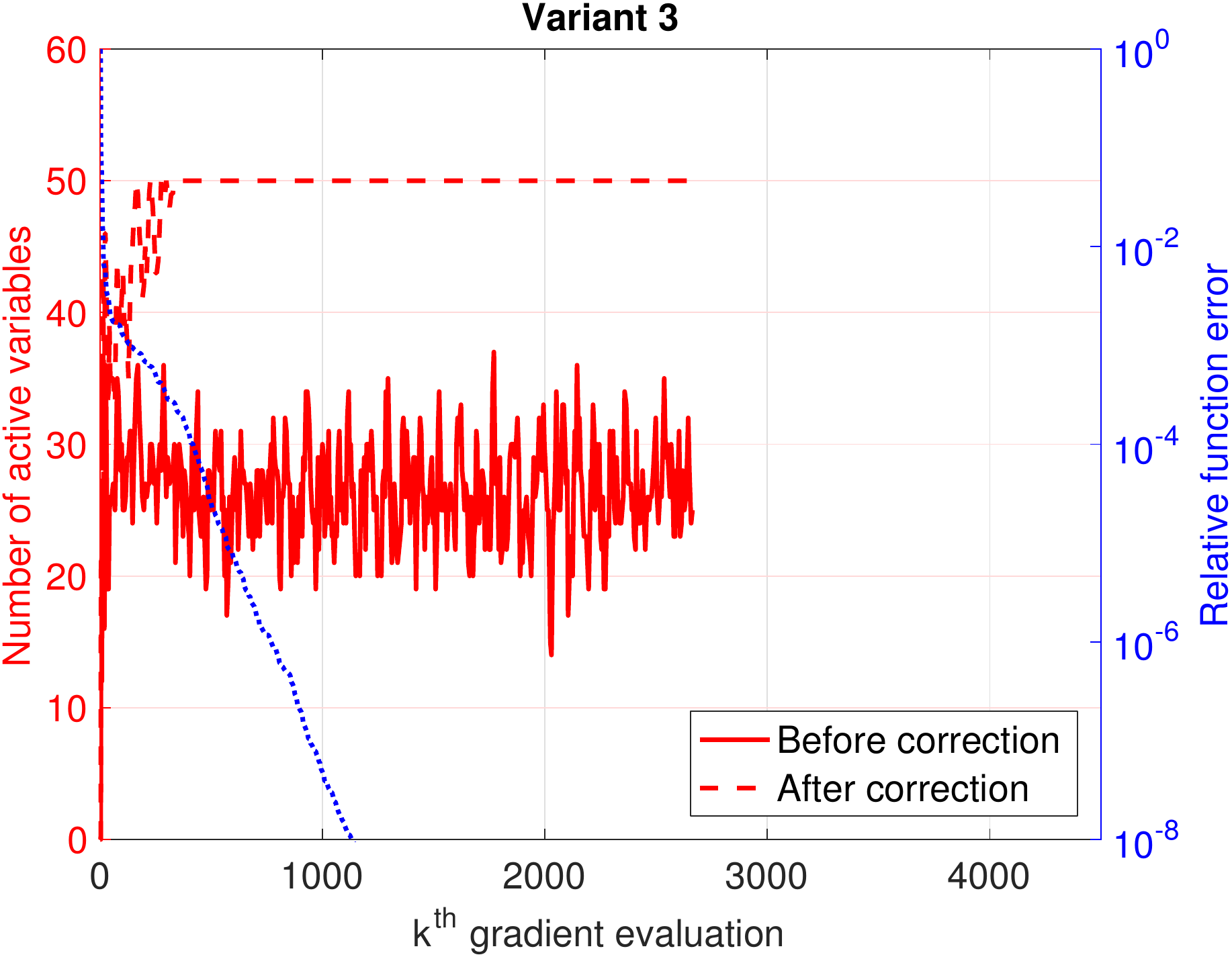}
\end{subfigure}\hspace*{\fill}
\begin{subfigure}{0.48\textwidth}
\includegraphics[width=\linewidth]{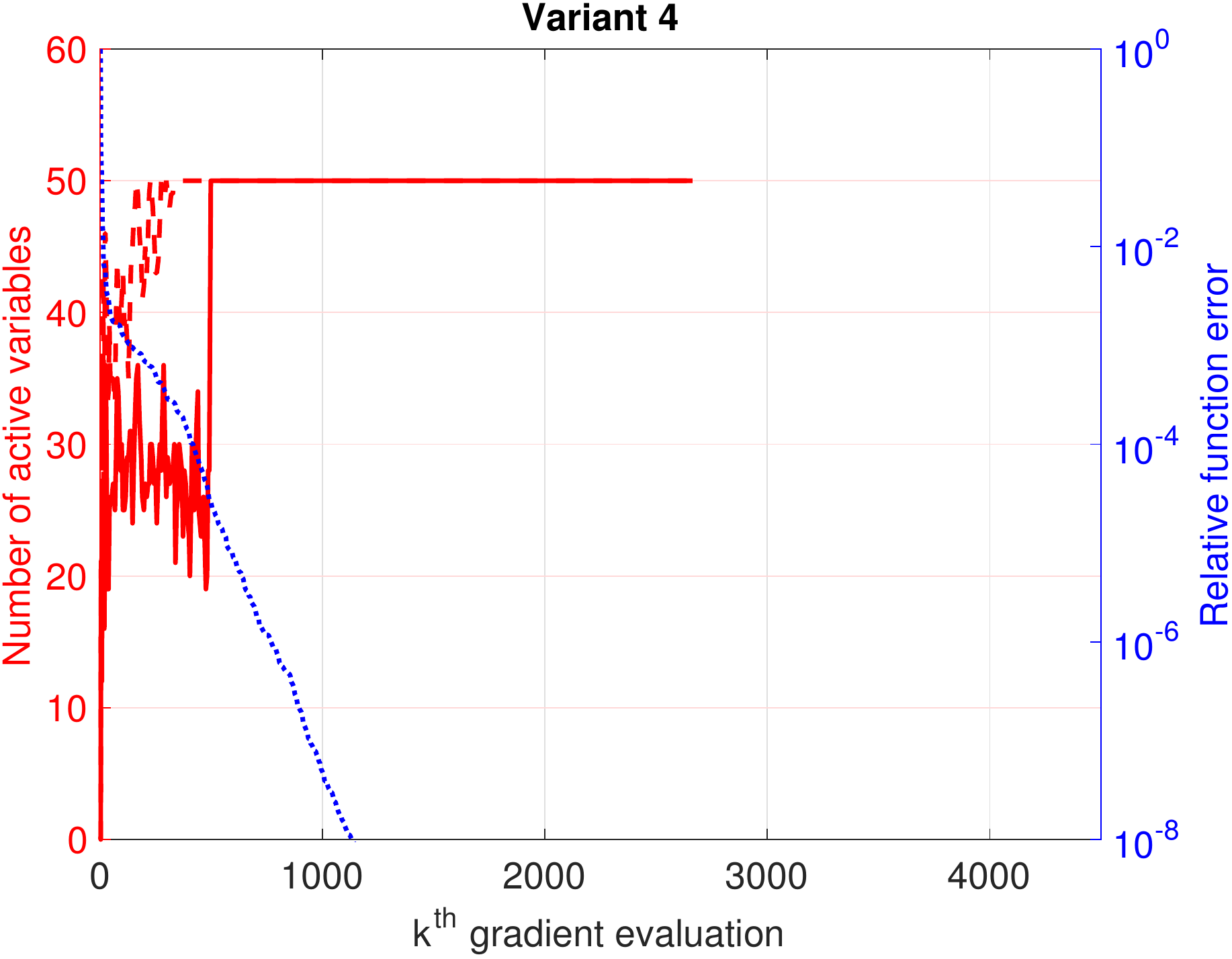}
\end{subfigure}
    \cprotect \caption{A comparison of the four variants of the algorithm on the \verb|Myopic_Decoupled| problem.}
    \label{fig:four-way-comparison-myopic}
\end{figure}

\begin{figure}[t] 
\begin{subfigure}{0.48\textwidth}
\includegraphics[width=\linewidth]{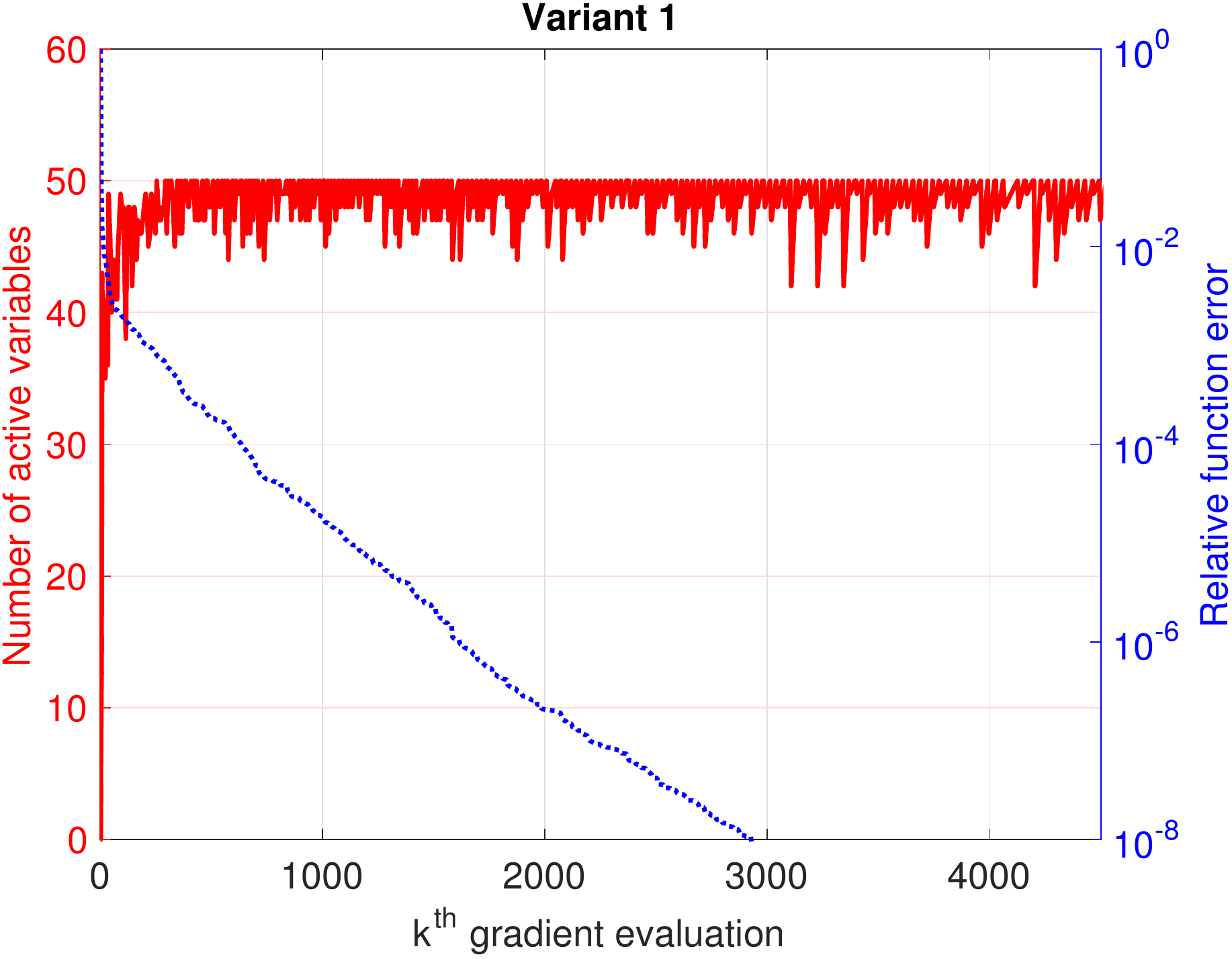}
\end{subfigure}\hspace*{\fill}
\begin{subfigure}{0.48\textwidth}
\includegraphics[width=\linewidth]{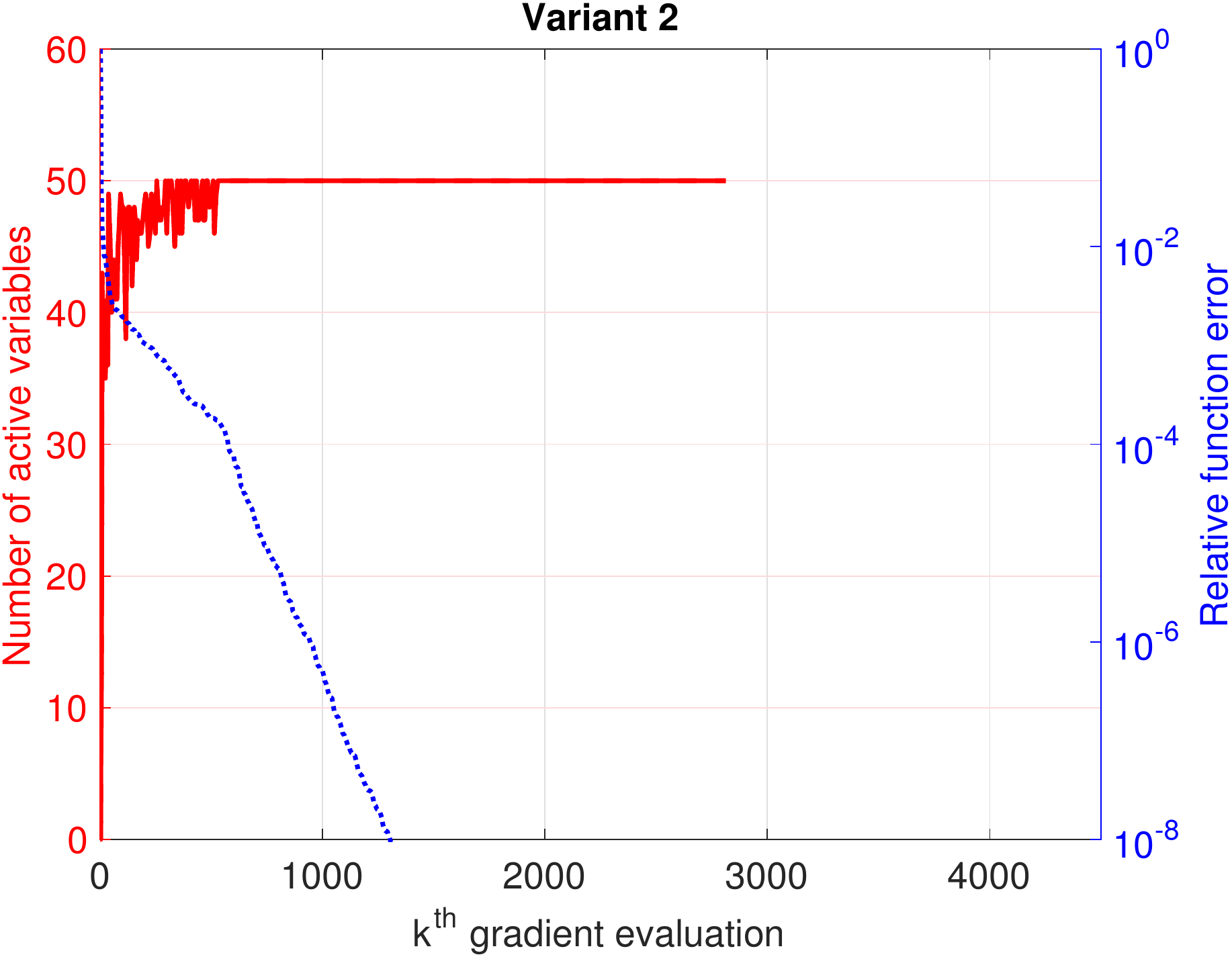}
\end{subfigure}

\medskip
\begin{subfigure}{0.48\textwidth}
\includegraphics[width=\linewidth]{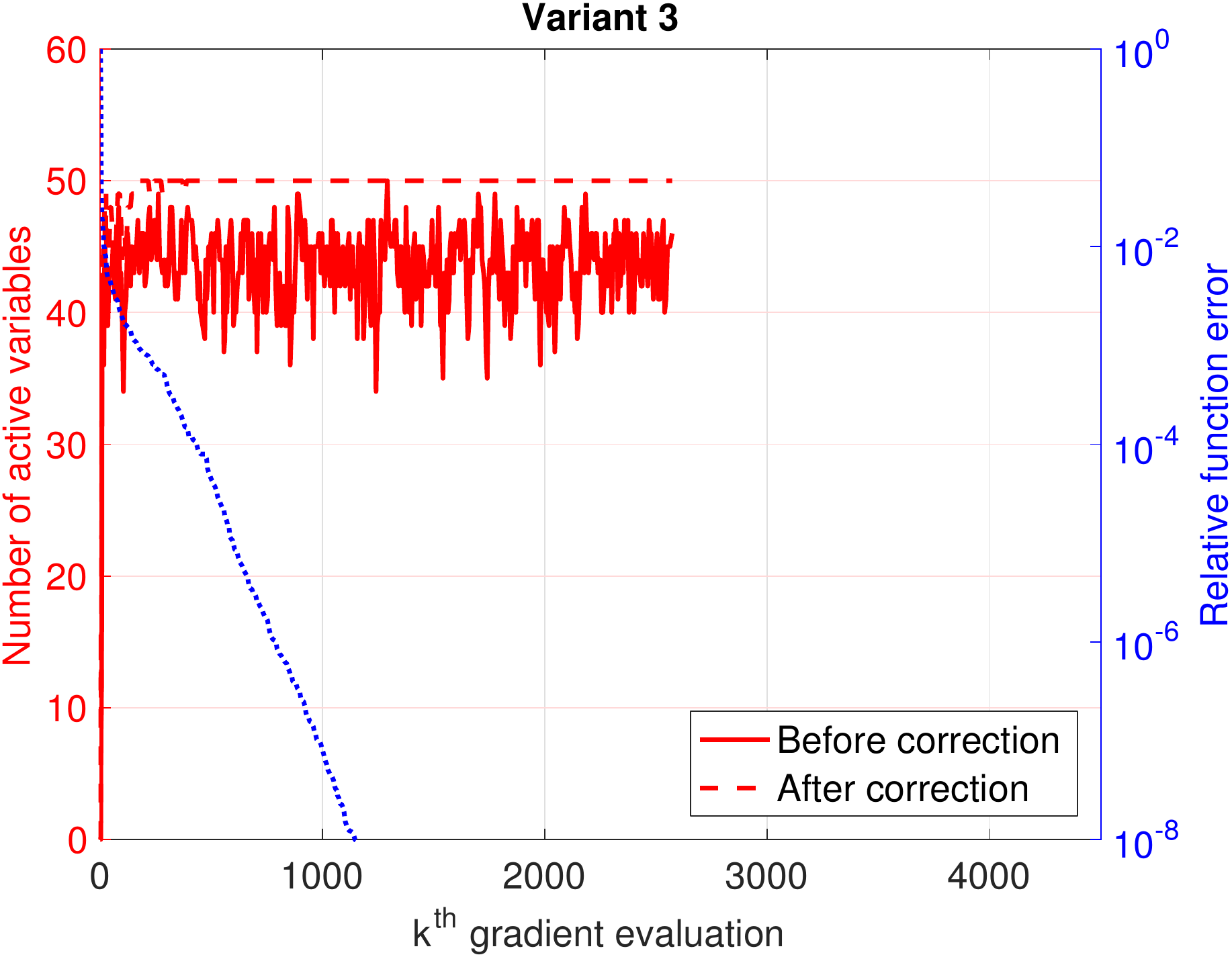}
\end{subfigure}\hspace*{\fill}
\begin{subfigure}{0.48\textwidth}
\includegraphics[width=\linewidth]{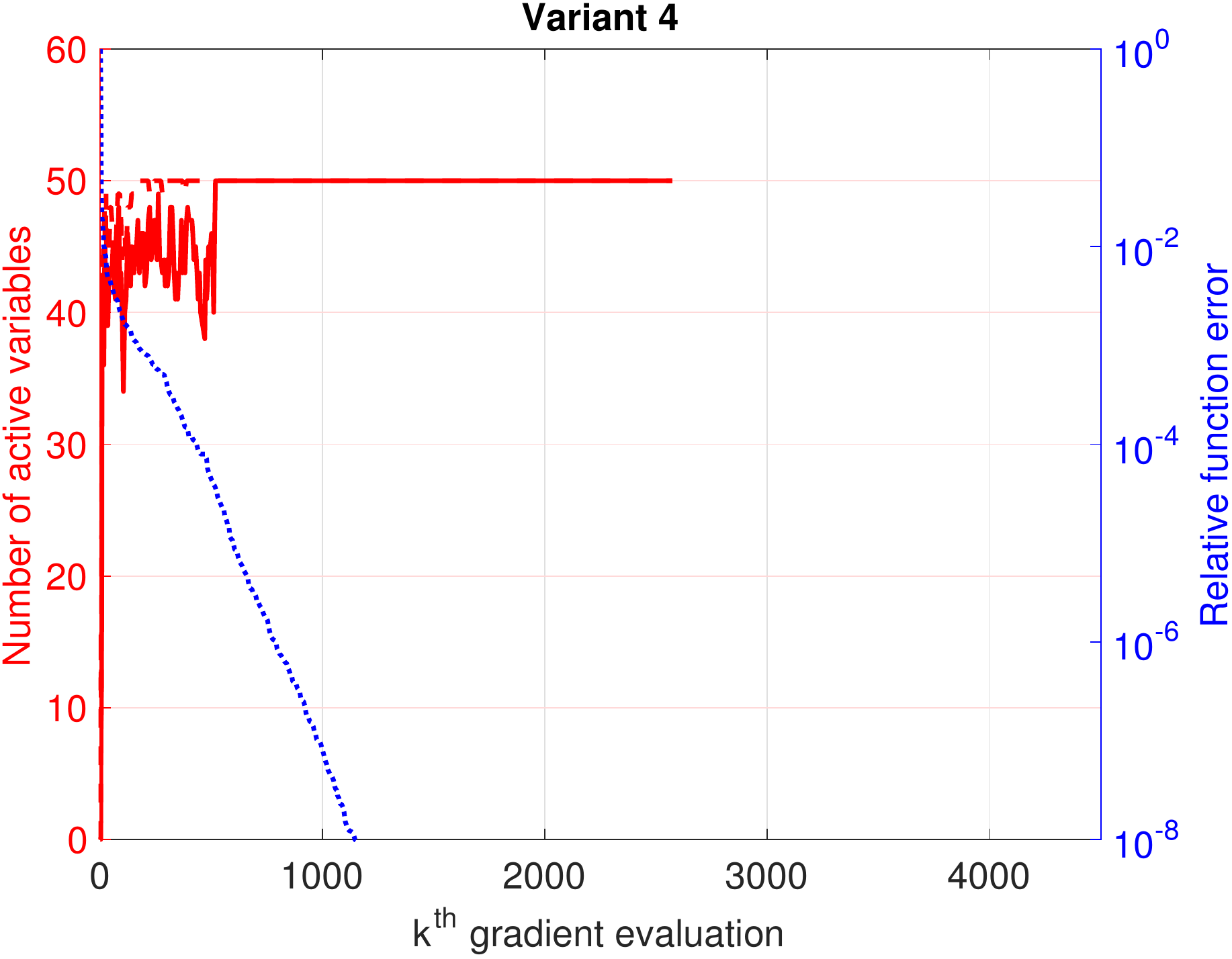}
\end{subfigure}
    \cprotect \caption{A comparison of the four variants of the algorithm on the \verb|Myopic_Coupled| problem.}
    \label{fig:four-way-comparison-myopic-coupled}
\end{figure}

Let us first consider the \texttt{Myopic\_Decoupled} and \texttt{Myopic\_Coupled} problems with $n=100$ for one particular starting point. These problems are designed specifically to highlight the failure of gradient-based active-set prediction strategies.  Figures~\ref{fig:four-way-comparison-myopic} and \ref{fig:four-way-comparison-myopic-coupled} detail the behavior of the algorithm over the course of the optimization.
In each plot, the dotted blue line depicts the progress in objective function, measured as $(f(x^k)-f(x^\star))/(f(x^0)-f(x^\star))$, where $x^\star$ is the optimal solution.
The solid red line gives the size of the (initial) active set $\AAA^k$.  For the variants that employ the active set correction strategy, the dashed red line gives the size of the active set at the end of Algorithm~\ref{alg:correction}.

As one might expect, Variant 1, which uses the myopic gradient and no correction strategy, fails to identify the optimal active set (which contains 50 variables) and its guess $\AAA^k$ keeps fluctuating.  The reduction in the objective function is significantly slower compared to the other variants.
When the subgradient approximation is used in Variant 2, the objective function decreases faster, and after a certain number of iterations, the active set settles to the optimal active set.
When the correction strategy is used in Variants 3 and 4, the optimal active set is identified more quickly, and the reduction in the objective is even faster.
In these experiments, there is very little difference in performance between the two initializations of $\AAA^{\textup{init}}$ in Algorithm~\ref{alg:correction}.
We observe similar behavior for larger dimensions of this problem. Further, rapid fluctuations in the active set are also seen in the other algorithms (L-BFGS-B, L-BFGS-B-NS and LMBM-B) which also employ gradient-based identification strategies.


\begin{figure} 
\begin{subfigure}{0.48\textwidth}
\includegraphics[width=\linewidth]{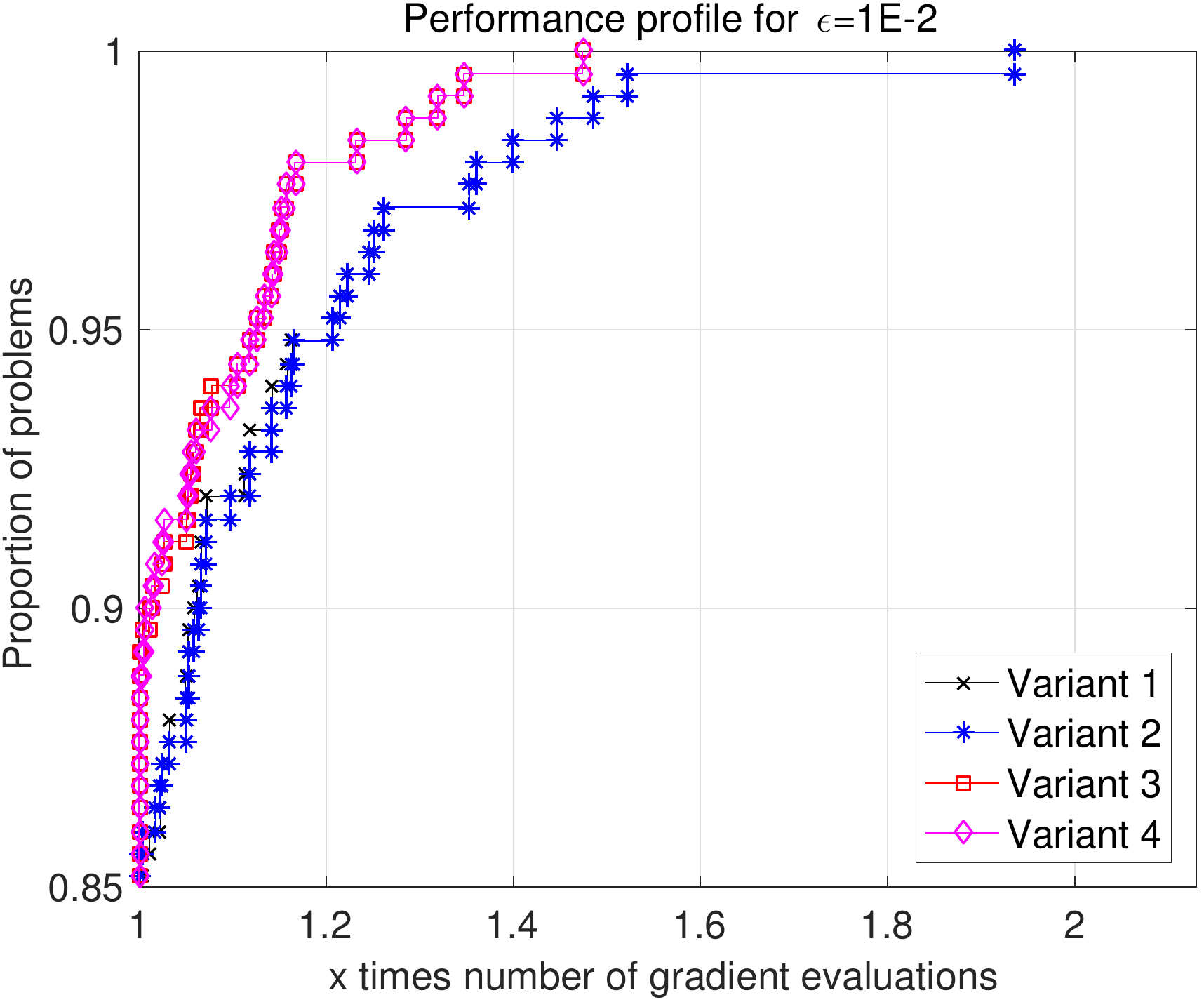}
\end{subfigure}\hspace*{\fill}
\begin{subfigure}{0.48\textwidth}
\includegraphics[width=\linewidth]{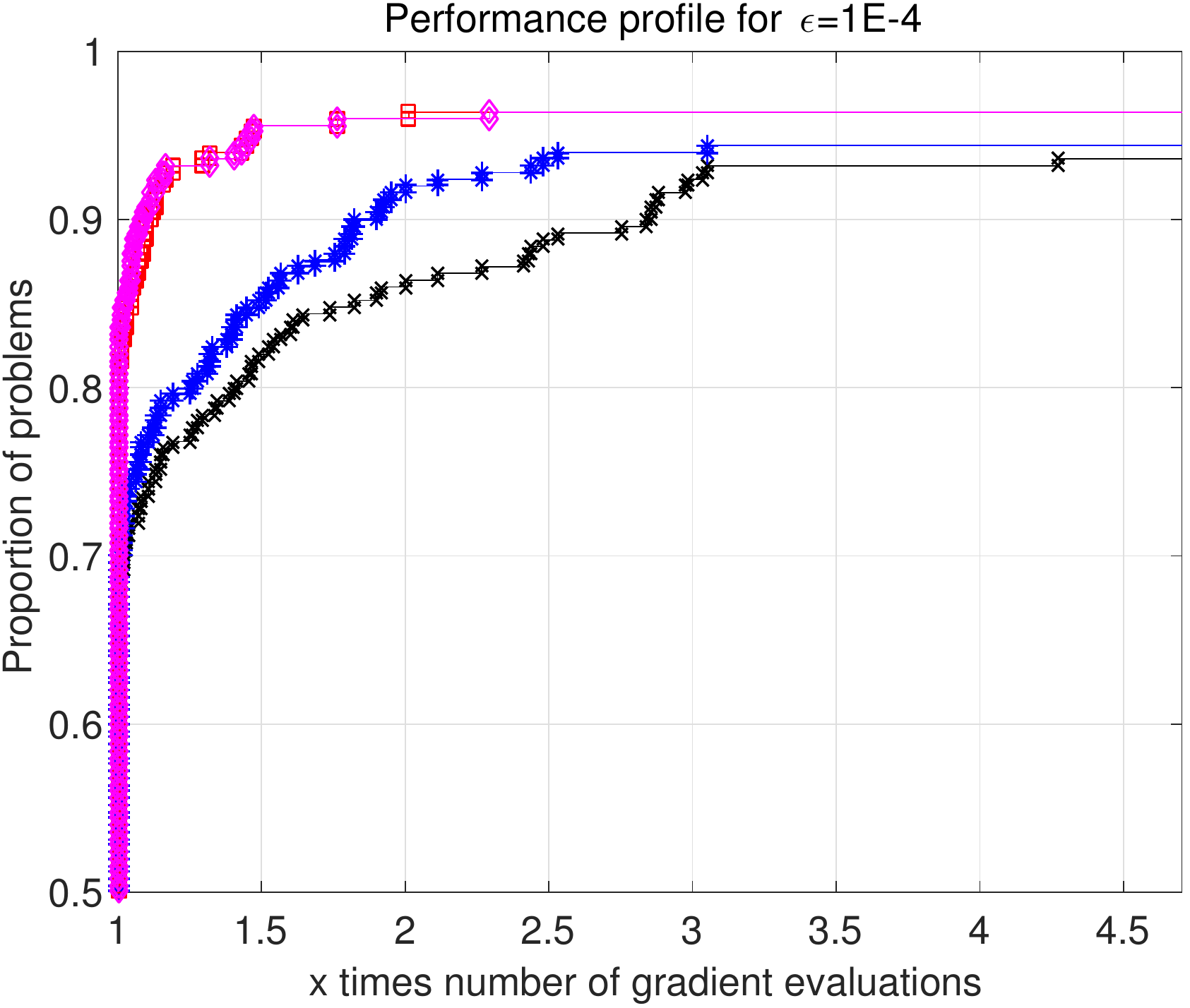}
\end{subfigure}

\medskip
\begin{subfigure}{0.48\textwidth}
\includegraphics[width=\linewidth]{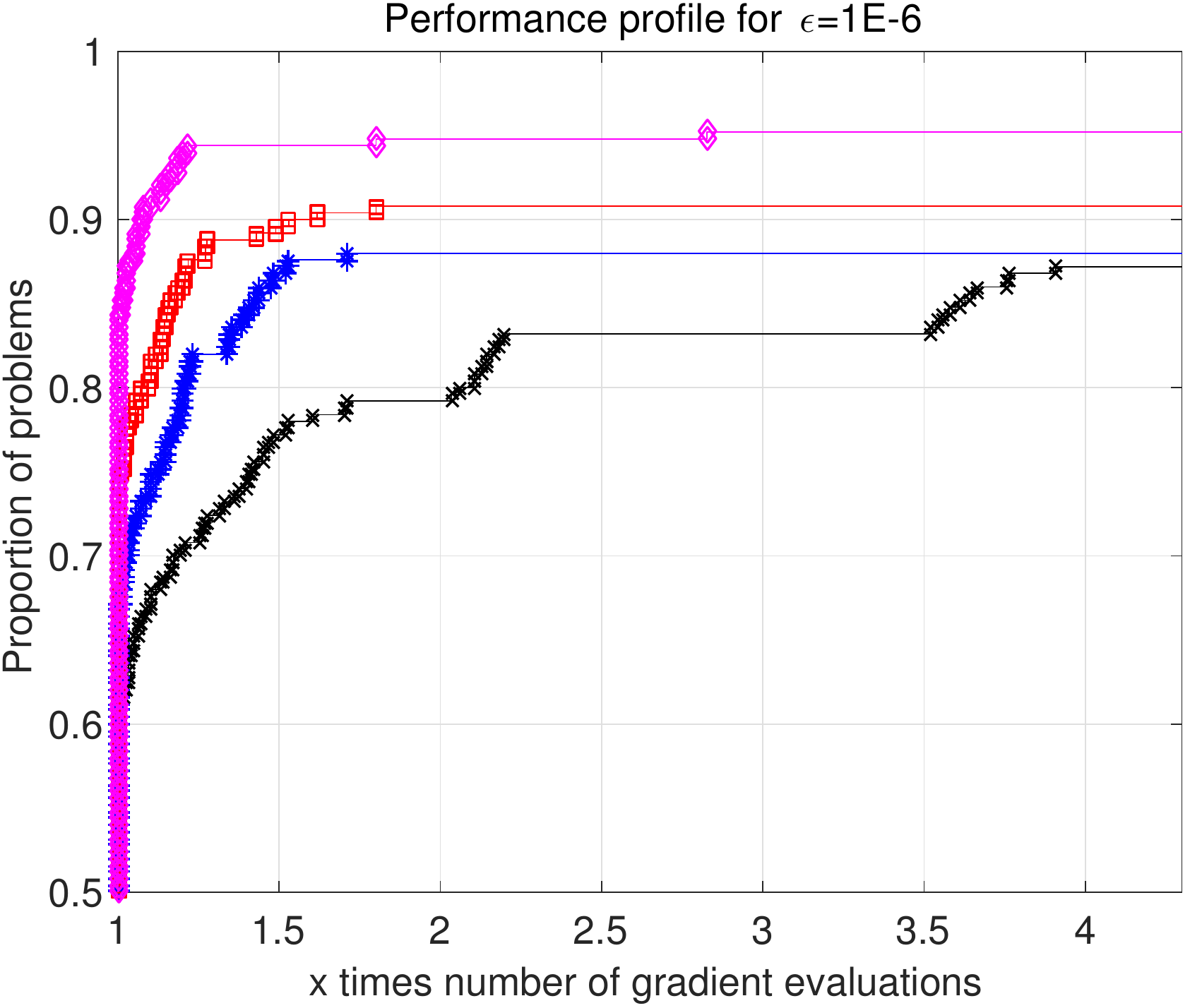}
\end{subfigure}\hspace*{\fill}
\begin{subfigure}{0.48\textwidth}
\includegraphics[width=\linewidth]{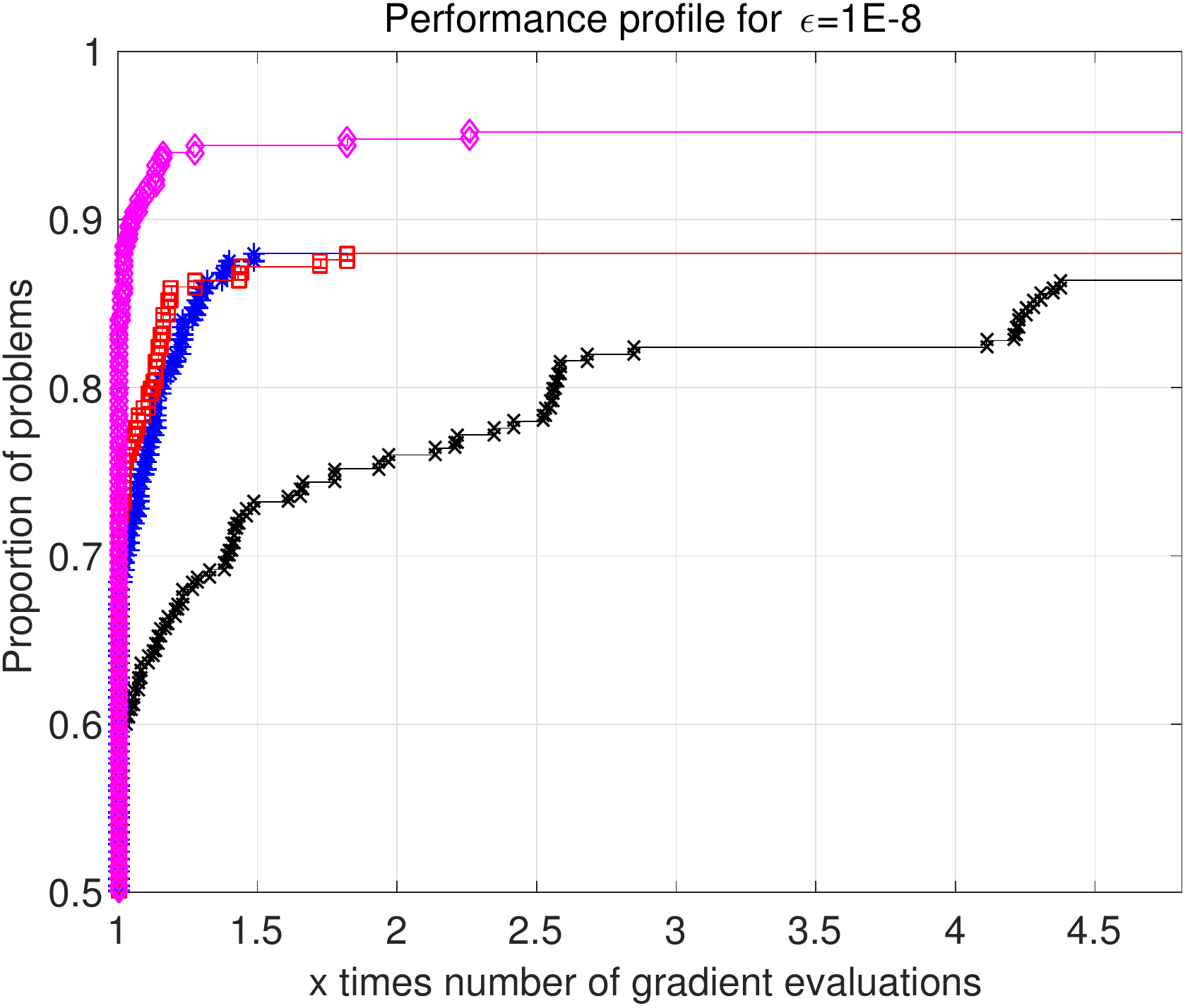}
\end{subfigure}

    \cprotect \caption{Dolan-Mor\'e performance profiles comparing the four variants of the algorithm on $250$ test problems for $\epsilon=10^{-2},10^{-4},10^{-6}$ and $\epsilon=10^{-8}$.}
\label{fig:four-way-comparison-all}
\end{figure}

Next we assess the relative performance of the different variants for the entire set of 250 instances with $n=100$. Figure~\ref{fig:four-way-comparison-all} presents Dolan-Mor\'e performance profiles \cite{dolan2002benchmarking} with respect to the number of gradient evaluations. These profiles rely on a condition to determine when a run is deemed converged. For this purpose, given a tolerance $\epsilon>0$, we use the test
\begin{align}
    \label{eq:termination}
    \frac{f(x^k)-f^\star}{f(x^0)-f^\star}&<\epsilon
\end{align} where $f^\star$ is the best value found by any of the methods for the same instance. We present plots for four values of $\epsilon$ viz., $10^{-2}, 10^{-4}$, $10^{-6}$ and $10^{-8}$. 

\begin{figure}
\centering
\includegraphics[width=0.75\textwidth]{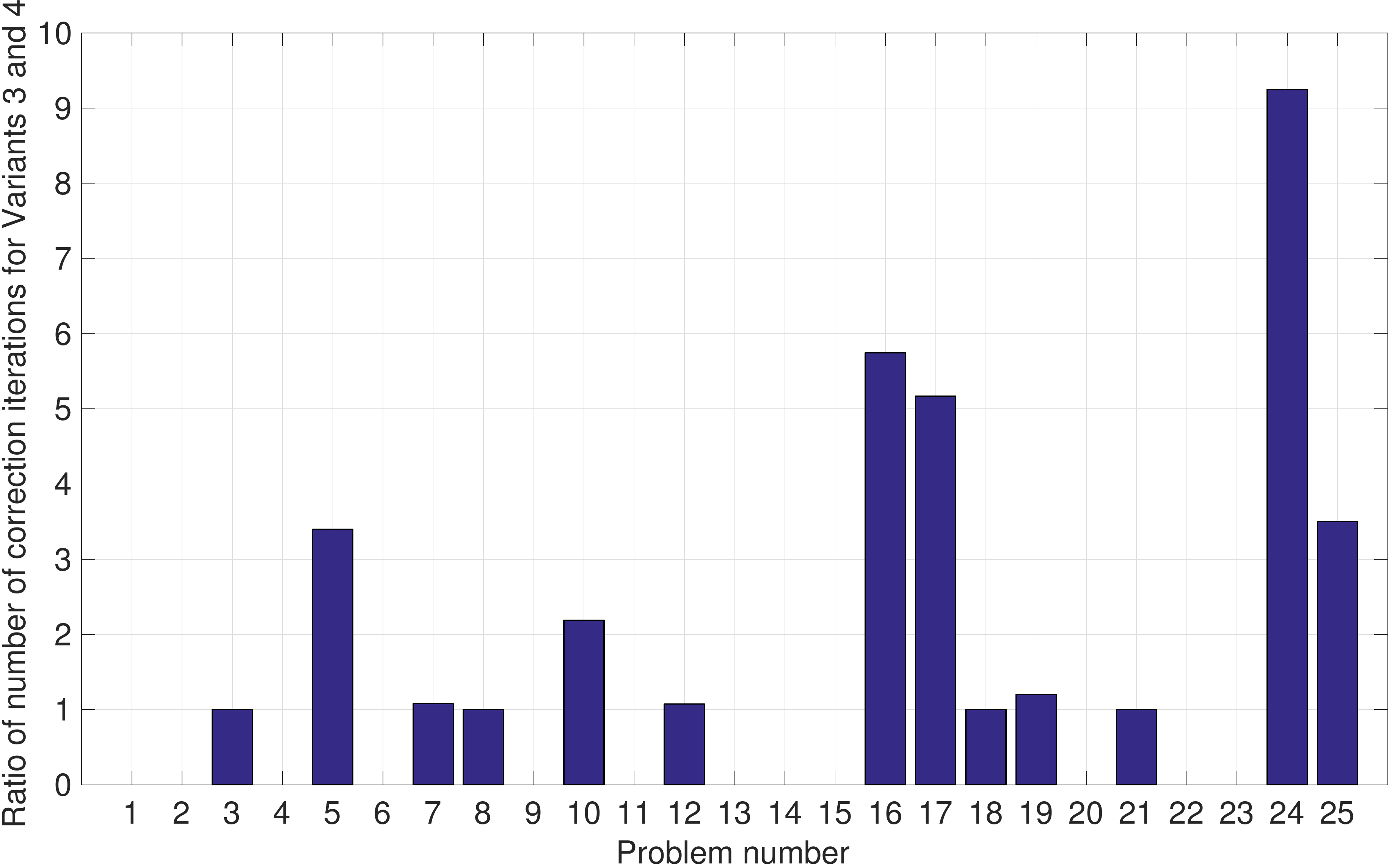}
\caption{Average ratio of number of corrections for Variants 3 and 4 for all problems for $\epsilon=10^{-4}$ and $n=100$. A ratio of $0$ indicates that both methods did not need any corrections.}
\label{fig:hist}
\end{figure}

This experiment reveals results similar to those found in Figures \ref{fig:four-way-comparison-myopic} and \ref{fig:four-way-comparison-myopic-coupled}. 
Variant 1 shows the worst behavior, and the use of the subgradient approximation in Variant 2 improves the convergence rate.  Using the corrective strategy gives the best performance, with an advantage for Variant 4 when a very tight tolerance is used.

Variants 3 and 4 incur different computational costs per iteration in Algorithm~\ref{alg:nqn}.  In addition to the cost of the corrective loop, Variant 4 relies on the solution of the quadratic program \eqref{eq:activesetprediction} for the computation of the subgradient approximation.  Figures \ref{fig:four-way-comparison-myopic} and \ref{fig:four-way-comparison-myopic-coupled} suggest that Variant 4 might require fewer iterations in the corrective loop in Algorithm~\ref{alg:correction} than Variant 3, since its initial active set $\AAA^{\textup{init}}$ is a better guess of the final active set returned by the correction procedure.
In Figure \ref{fig:hist}, we present the average ratio of the number of correction iterations for Variants 3 over 4.
Indeed, Variant 3 needs up to 10 times as many correction iterations as Variant 4 to achieve similar performance. 

Nevertheless, since the solution of the quadratic program \eqref{eq:activesetprediction} comes at a significant computational cost, we used Variant 3 for the remaining experiments.  Also, Variant 3 is consistent with Lemma~\ref{lem:p0}, so that we would encounter a zero step from the correction loop only when the current iterate is already stationary.

\subsection{Comparison with Other Methods}
\label{sec:experiments-academic}

\begin{figure} 
\begin{subfigure}{0.48\textwidth}
\includegraphics[width=\linewidth]{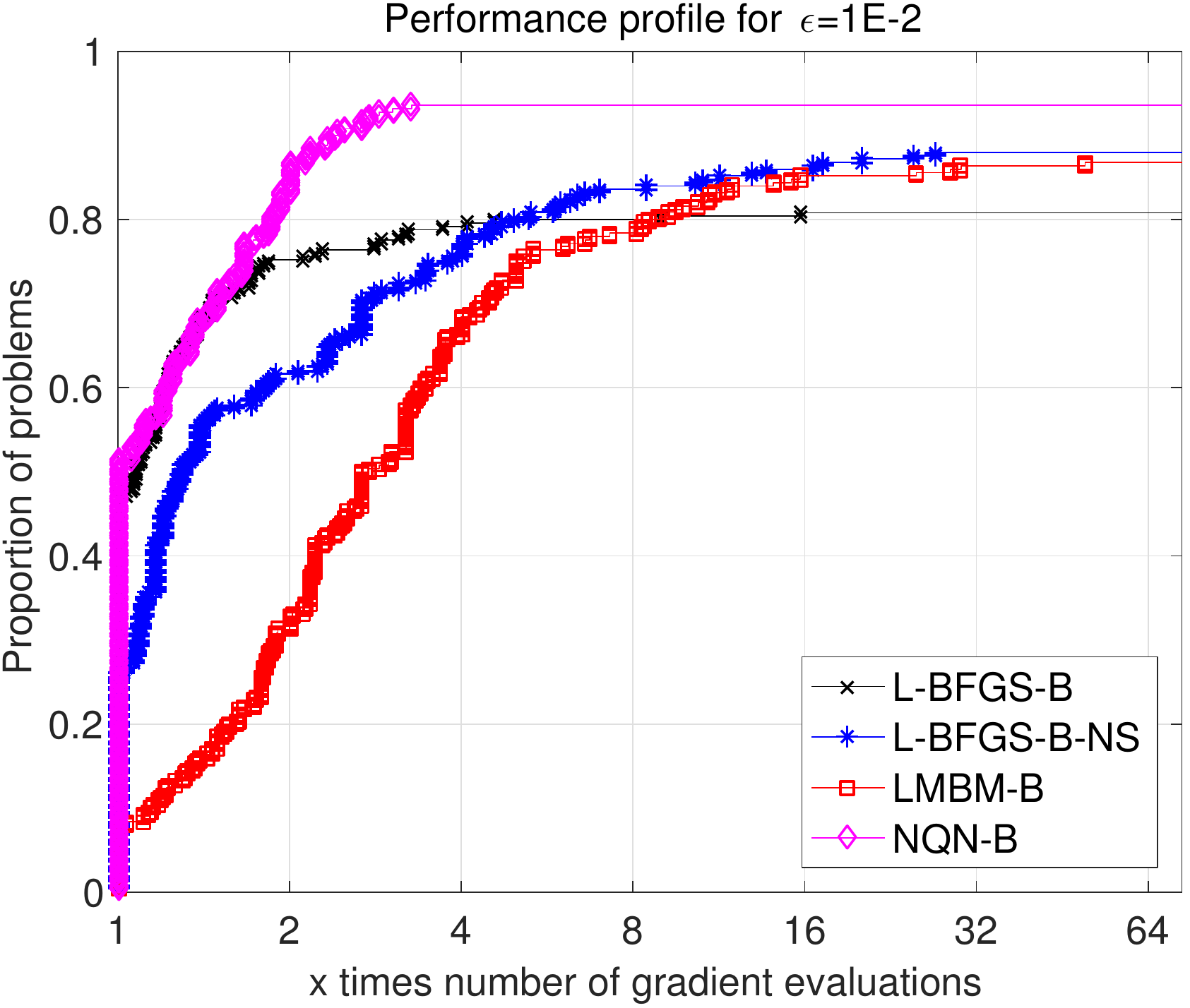}
\end{subfigure}\hspace*{\fill}
\begin{subfigure}{0.48\textwidth}
\includegraphics[width=\linewidth]{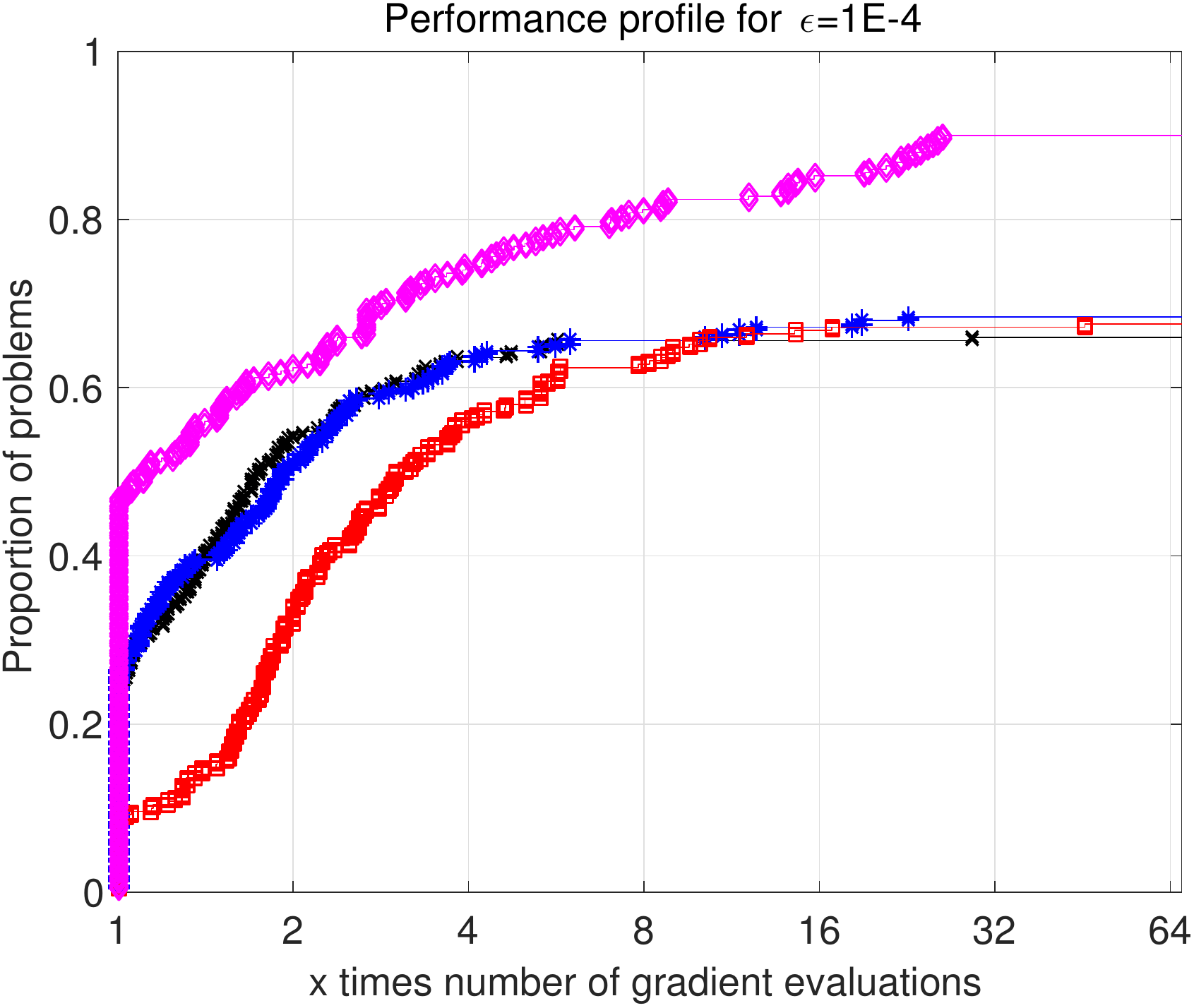}
\end{subfigure}
\caption{Dolan-Mor\'e performance profiles of gradient evaluations for $250$ test problems for $\epsilon=10^{-2}$ and $\epsilon=10^{-4}$ with $n=100$. }
    \label{fig:dolanmore_N100}
\end{figure}

\begin{figure} 
\begin{subfigure}{0.48\textwidth}
\includegraphics[width=\linewidth]{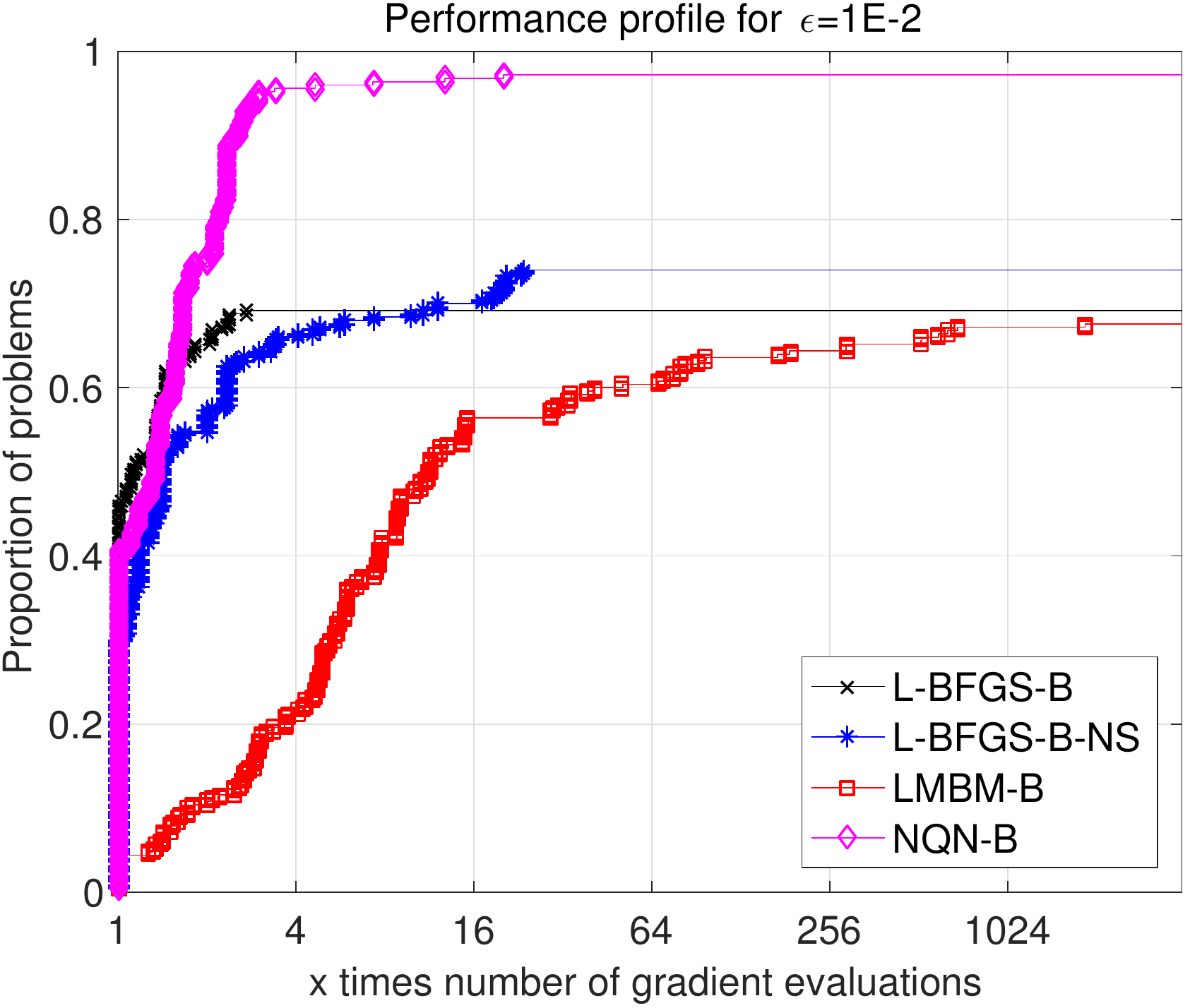}
\end{subfigure}\hspace*{\fill}
\begin{subfigure}{0.48\textwidth}
\includegraphics[width=\linewidth]{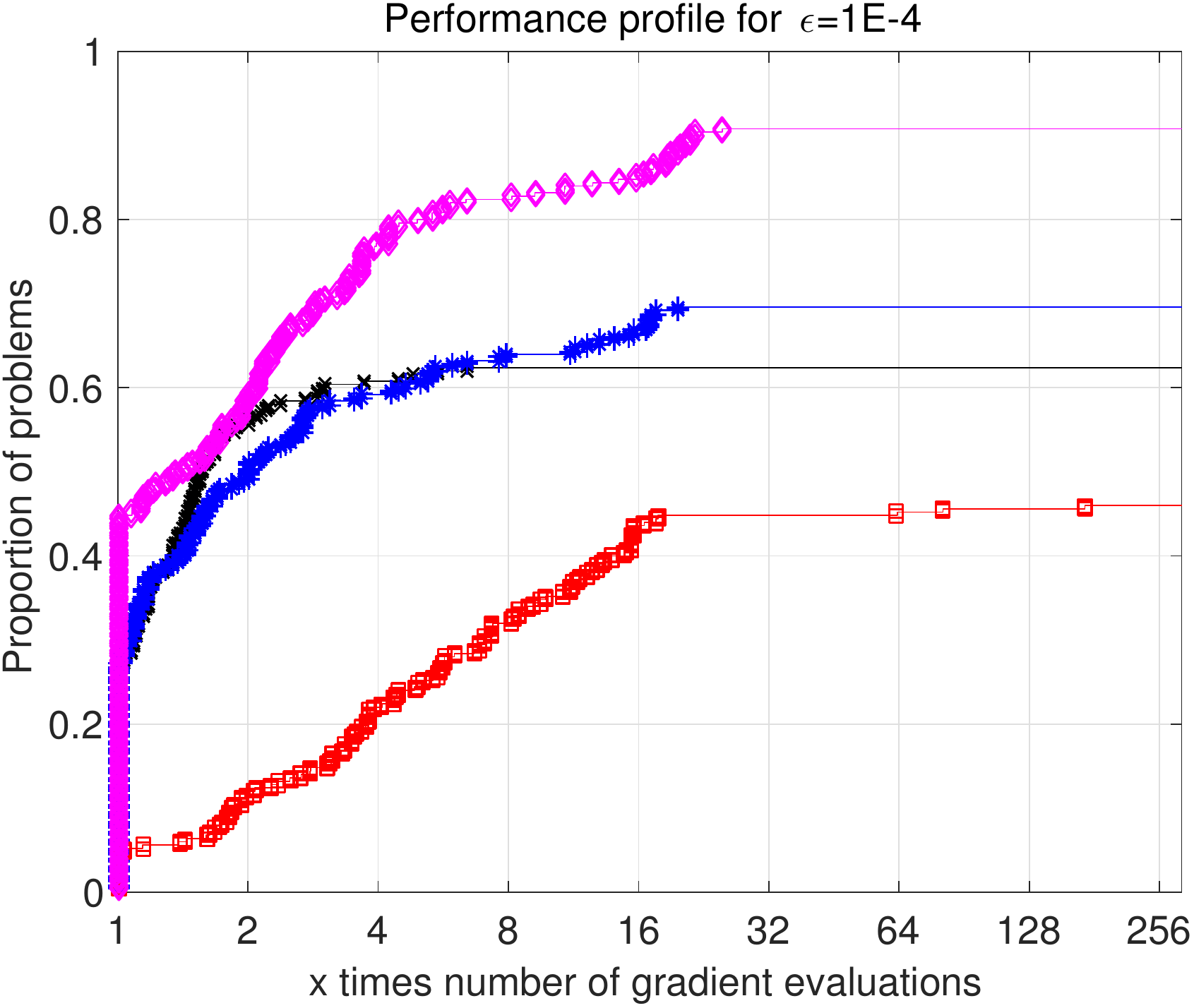}
\end{subfigure}
\caption{Dolan-Mor\'e performance profiles of gradient evaluations for $250$ test problems for $\epsilon=10^{-2}$ and $\epsilon=10^{-4}$ with $n=1000$. }
    \label{fig:dolanmore_N1000}
\end{figure}

\begin{figure} 
\begin{subfigure}{0.48\textwidth}
\includegraphics[width=\linewidth]{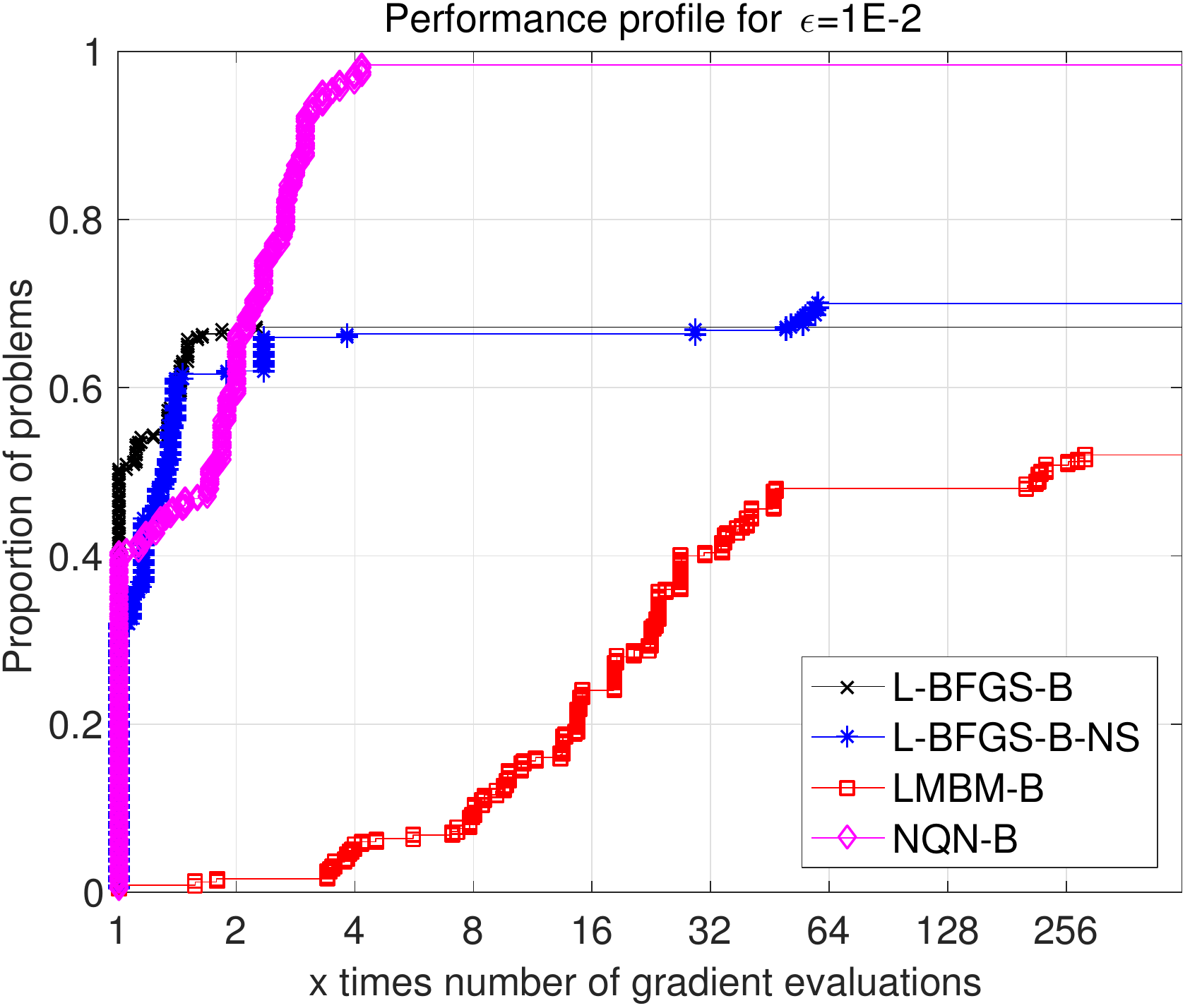}
\end{subfigure}\hspace*{\fill}
\begin{subfigure}{0.48\textwidth}
\includegraphics[width=\linewidth]{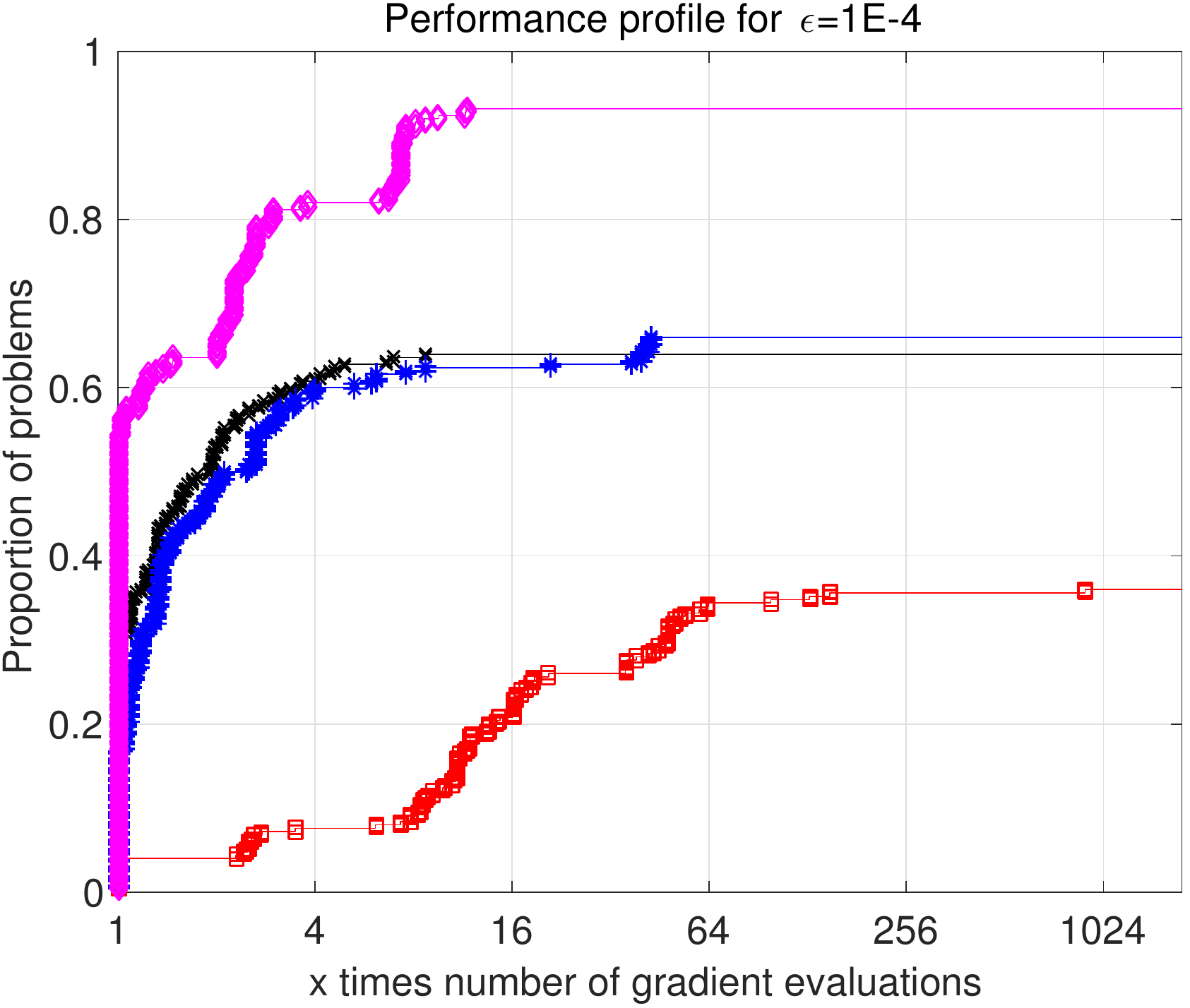}
\end{subfigure}
\caption{Dolan-Mor\'e performance profiles of gradient evaluations for $250$ test problems for $\epsilon=10^{-2}$ and $\epsilon=10^{-4}$ with $n=10000$. }
    \label{fig:dolanmore_N10000}
\end{figure}

We now compare \algoname with L-BFGS-B, L-BFGS-B-NS, and LMBM-B on the 250 test instances. Figures \ref{fig:dolanmore_N100}, \ref{fig:dolanmore_N1000}, and \ref{fig:dolanmore_N10000} correspond to three sets of experiments, with $n=100$, $n=1000$, and $n=10000$, respectively. We present performance profiles for two values of $\epsilon$, viz. $10^{-2}$ and $10^{-4}$. We do not report experiments for $\epsilon=10^{-6}$ and $\epsilon=10^{-8}$ since the relative error is based on the best function value obtained by any of the methods.  Tighter tolerances for $\epsilon$ would magnify insignificant differences between methods when neither are very close to an optimal solution. 
It can be seen that the proposed algorithm performs better than the other methods across different tolerances $\epsilon$ and problem sizes $n$. The figures show that \algoname is able to find a lower objective on more problems and requires fewer gradient evaluations. 
This difference is particularly pronounced for tight tolerances and large problem sizes.

In Table \ref{tab:failures}, we summarize the occurrence for failures of the various methods for tolerances of $\epsilon=10^{-2}$ and $\epsilon=10^{-4}$. The flag \OK indicates that the termination criterion was satisfied at some iteration, \MAX corresponds to reaching maximum number of gradient evaluations, and \OTHER implies other failures which include solver-specific causes such as spurious termination of the line search, numerical issues, or convergence to a non-stationary point. 
For \algoname we break down the number of \OTHER failures into convergence to a point with no feasible direction in step~\ref{alg:zero_step} of Algorithm \ref{alg:generalized-wolfe} (first number) and line search failure in step~\ref{alg:error} of Algorithm \ref{alg:generalized-wolfe} (second number).
%

\begin{table}
    \centering
    \begin{tabular}{|c||c|c|c||c|c|c|}
    \hline 
        Flag &  \OK & \MAX & \OTHER &  \OK & \MAX & \OTHER \\ \hline
          & \multicolumn{3}{|c||}{$\epsilon=10^{-2}$} & \multicolumn{3}{c|}{$\epsilon=10^{-4}$}  \\ \hline
        
        \multicolumn{7}{|c|}{$n=100$} \\ \hline
L-BFGS-B & 202 & 0 & 48 & 165 & 0 & 85 \\
L-BFGS-B-NS & 220 & 25 & 5 & 171 & 58 & 21 \\
LMBM-B & 217 & 16 & 17 & 169 & 64 & 17 \\
NQN & 234 & 12 & 4 + 0 & 225 & 17 & 7 + 1 \\ \hline   

        \multicolumn{7}{|c|}{$n=1000$} \\ \hline
L-BFGS-B & 173 & 4 & 73 & 156 & 8 & 86 \\
L-BFGS-B-NS & 185 & 61 & 4 & 174 & 68 & 8 \\
LMBM-B & 169 & 11 & 70 & 115 & 54 & 81 \\
NQN & 243 & 6 & 1 + 0 & 227 & 22 & 1 + 0 \\ \hline
        \multicolumn{7}{|c|}{$n=10000$} \\ \hline
L-BFGS-B & 168 & 17 & 65 & 160 & 21 & 69 \\
L-BFGS-B-NS & 175 & 73 & 2 & 165 & 82 & 3 \\
LMBM-B & 130 & 30 & 90 & 90 & 70 & 90 \\
NQN & 246 & 2 & 2 + 0 & 233 & 15 & 2 + 0 \\ \hline

\end{tabular}
    \caption{Number of outcomes with different termination messages.}
    \label{tab:failures}
\end{table}

As can be seen from Table \ref{tab:failures}, failures for \algoname are more often due to budget exhaustion rather than another type of failure.  In total, there were 10 instances in which numerical issues led to a bad search direction.  A line search error was observed only once.
The cause for budget exhaustion in \algoname is, in part, due to the tight tolerance of $\epsilon_{\textup{abs}}$; close to a solution, the bracketing procedure takes many iterations in order to find points providing sufficient function decrease. 
Most of the large number of failures for L-BFGS-B occur due to a breakdown in the line search. This is to be expected since L-BFGS-B employs a strong Wolfe line search which is difficult to be satisfied with a nonsmooth objective.  When the weak Wolfe line search is used in L-BFGS-B-NS instead, the number of line search failures is reduced significantly.  Nevertheless, the overall number of successfully solved problems increases only marginally.
LMBM-B is the least robust method, with a noticeable increase in the failure rate as the problem size grows.





\section*{Acknowledgements}
The first author was partially supported by Office of Naval Research grant N00014-14-1-0313. The second author was partially supported by National Science Foundation grant DMS-1522747. The authors are grateful to Jorge Nocedal and Michael Overton for their insightful comments.

\bibliographystyle{plain}
\bibliography{references}

\newpage
\end{document}